\documentclass[11pt, reqno]{amsart}
\usepackage{amsmath} 
\usepackage{amsthm}
\usepackage{graphicx}
\usepackage{ amssymb }
\usepackage{amsfonts}
\usepackage{cases}
\usepackage[margin=1in]{geometry}
\usepackage{bm}
\usepackage{calc}
\usepackage{hyperref}
\usepackage{cleveref}
\usepackage{lipsum}
\usepackage{color} 
\usepackage{cancel}
\usepackage{pgfplots}
\pgfplotsset{compat=1.15}
\usepackage{mathrsfs}
\usepackage{tikz}
\usetikzlibrary{arrows}
\usepackage{url}

\let\OLDthebibliography\thebibliography
\renewcommand\thebibliography[1]{
  \OLDthebibliography{#1}
  \setlength{\parskip}{0pt}
  \setlength{\itemsep}{0pt plus 0.3ex}
}

\hypersetup{
    colorlinks=true, 
    linktoc=all,     
    linkcolor=blue,  
    citecolor = red
}

\newtheorem{Proposition}{Proposition}[section]
\newtheorem{Lemma}[Proposition]{Lemma}

\newtheorem{Theorem}[Proposition]{Theorem}

\newtheorem{Corollary}[Proposition]{Corollary}

\newtheorem{Remark}[Proposition]{Remark}

\newtheorem{Definition}[Proposition]{Definition}

\newtheorem{MainTheorem}{Theorem}

\newcommand{\Z}{\ensuremath{\mathbb{Z}}}
\newcommand{\N}{\ensuremath{\mathbb{N}}}
\newcommand{\R}{\ensuremath{\mathbb{R}}}
\DeclareMathOperator*{\esssup}{ess\,sup}

\newcommand{\LP}{\ensuremath{\mathbb{P}}} 
\newcommand\norm[2]{\left\Vert#1\right\Vert_{#2}} 
\newcommand\Bes[3]{\dot{B}^{#1}_{{#2},{#3}}}

\definecolor{qqqqff}{rgb}{0,0,1}
\definecolor{ffqqqq}{rgb}{1,0,0}
\definecolor{qqwuqq}{rgb}{0,0.59215686274509803,0}
\definecolor{ffzztt}{rgb}{1,0.6,0.2}
\definecolor{uuuuuu}{rgb}{0.26666666666666666,0.26666666666666666,0.26666666666666666}
\definecolor{ffdxqq}{rgb}{1,0.8431372549019608,0}
\definecolor{xfqqff}{rgb}{0.4980392156862745,0,1}
\definecolor{qqzzff}{rgb}{0,0.6,1}

\title[On rough Calderón solutions and applications to the singular set]{On rough Calderón solutions to the Navier-Stokes equations and applications to the singular set}
\author{Henry Popkin}
\thanks{HP is supported by Raoul \& Catherine Hughes (Alumni funds) and the University Research Studentship award EH-MA1333.}
\date{\today}

\begin{document}

\maketitle
\begin{center}
University of Bath\\
Department of Mathematical Sciences\\
BA2 7AY\\
hp775@bath.ac.uk
\end{center}
\begin{abstract}
In 1934, Leray \cite{Leray} proved the existence of global-in-time weak solutions to the Navier-Stokes equations for any divergence-free initial data in $L^2$. In the 1980s, Giga and Kato \cite{Giga, Kato84} independently showed that there exist global-in-time mild solutions corresponding to small enough critical $L^3(\R^3)$ initial data. In 1990, Calderón \cite{Calderón} filled the gap to show that there exist global-in-time weak solutions for all supercritical initial data in $L^p$ for $2< p<3$ by utilising a splitting argument, blending the constructions of Leray and Giga-Kato.
\par In this paper, we utilise a ``Calderón-like" splitting to show the global-in-time existence of weak solutions to the Navier-Stokes equations corresponding to supercritical Besov space initial data $u_0 \in \Bes{s}{q}{\infty}$ where $q>2$ and $-1+\frac{2}{q}<s<\min \left(-1+\frac{3}{q},0 \right)$, which fills a similar gap between Leray and known mild solution theory in the Besov space setting. We also use the Calderón-like splitting to investigate the structure of the singular set under a Type-I blow-up assumption in the Besov space setting, which is considerably rougher than in previous works. 
\end{abstract}
\pagenumbering{roman}
\setcounter{tocdepth}{1}
\tableofcontents

\pagenumbering{arabic}
\newpage
\section{Introduction}

In this paper, we consider the three-dimensional incompressible Navier-Stokes equations on $\R^3\times (0,T)$ for some $0<T\leq \infty$: 
\begin{equation}\label{Navier-Stokes equations}
    \partial_t u -\Delta u +(u \cdot \nabla)u + \nabla \pi =0, \qquad\qquad  \text{div}(u) = 0. \tag{NSE}
\end{equation}
These equations enjoy the following rescaling property: if $(u, \pi)$ is a solution to the Navier Stokes equations with initial data $u_0$, then the rescaled functions \begin{align} \label{NSrescale} u_\lambda(x,t) := \lambda u(\lambda x,\lambda^2t) \qquad \text{and} \qquad  \pi_\lambda(x,t):= \lambda^2\pi(\lambda x,\lambda^2t)
\end{align}
solve the Navier-Stokes equations with initial data $u_{0,\lambda}(x) := \lambda u_0(\lambda x) $. It is known that the criticality\footnote{See Subsection \ref{criticaldefn} for the definition of the criticality of a scalar quantity.} of certain quantities is closely related to the existence and regularity theory for (\ref{Navier-Stokes equations}). For example, given divergence-free\footnote{Initial data will always be divergence-free (in the sense of distributions).} initial data in the subcritical spaces $H^s(\R^3)$ for $s>\frac{d}{2}-1$, $L^p(\R^3)$ for $p>3$ or $\Bes{s}{p}{\infty}(\R^3)$ for $p> 3$ and $-1+\frac{3}{p}< s<0$, there exist mild solutions, which are smooth, but are only known to exist for a short time - see for example, \cite{Kato64, FJR72}. 

On the other hand, in Leray's 1934 seminal paper \cite{Leray}, it is shown that if we are given initial data $u_0\in L^2(\R^3)$, there exists a \textit{weak Leray-Hopf solution}\footnote{This solution satisfies \ref{Navier-Stokes equations} in the sense of distributions. It also satisfies the energy inequality. $L^2(\R^3)$ is a natural class to study as it may physically be viewed the space of initial data with finite kinetic energy. This is a supercritical space.} to (\ref{Navier-Stokes equations}) which is actually global-in-time. However, this is at the potential cost of uniqueness or regularity; for example, the recent paper \cite{nonuniqueness} demonstrates the non-uniqueness of Leray-Hopf weak solutions to the \textit{forced} (\ref{Navier-Stokes equations}) for a specific forcing in $L^1_tL^2_x(\R^3_+ \times (0,T))$. Later, Giga and Kato \cite{Giga, Kato84} independently showed the existence of global-in-time mild solutions corresponding to sufficiently small initial data in the critical space $L^3(\R^3)$, and Calderón \cite{Calderón} then showed that there exist global-in-time weak solutions corresponding to any initial data in $L^p(\R^3)$ for $p\in (2,3)$, filling the supercritical gap between the previous two results. We also mention that Leray's result has been generalised to include all initial data in $L^2_{uloc}(\R^3)$ with certain decay properties \cite{l2ulocsolns, Kwonslowdecay}. One may ask: \begin{enumerate}
    \item[\textbf{(Q1)}] What other classes of initial data give rise to global-in-time weak solutions to (\ref{Navier-Stokes equations})?
\end{enumerate} 
In particular, with access to the mild theory for subcritical Besov space initial data and taking inspiration from Calderón, we fill the supercritical gap in the Besov space setting in Theorem \ref{Global solution supercritical Besov data}.

Let us now turn our attention back to Leray-Hopf weak solutions. Given sufficiently smooth initial data, it is well known that a Leray-Hopf weak solution is smooth and unique for a short time; if such a solution first loses smoothness at time $T>0$, it is known that the solution necessarily forms a singularity\footnote{See Subsection \ref{singularsection} for a precise defintion.}. Hence, it is natural to ask whether a weak Leray-Hopf solution is smooth for \textit{all time}, or whether singularities may form in finite-time. We mention that there has been recent numerical evidence that singularities may occur for the axisymmetric Navier-Stokes equations \cite{computerblowup}. Leray conjectured the existence of self-similar blow-up solutions, which were ruled out in \cite{selfsimNeustupa} and \cite{selfsimTsai}. This motivated other natural scale-invariant blow-up ansatz to be studied, such as backwards discretely self-similar solutions (DSS) (see, for example, \cite{Chae2016RemovingDS}), or more generally Type-I solutions\footnote{Loosely, 
Type-I solutions are solutions that possess a critical bound in a neighbourhood of a singularity. For a formal definition see \cite{AlbBarkerLouiville}.}. To gain insight into potential Type-I singularity formation, we study the structure of the singular set for Type-I solutions. We investigate the following question:
\begin{enumerate}
    \item[\textbf{(Q2)}] What is the structure of the singular set of a solution to (\ref{Navier-Stokes equations}), supposing that it first loses smoothness at a time $T>0$ whilst satisfying some scale-invariant bound?
\end{enumerate}
\pagebreak
\subsection{Statement of results}
Our first main Theorem addresses \textbf{(Q1)}. Just as Calderón \cite{Calderón} filled the supercritical gap to find solutions corresponding to initial data in $L^p$ for $2<p<3$, between Leray ($L^2(\R^3)$) and Kato's mild solution theory (small enough in $L^3(\R^3)$), we fill a similar gap in the Besov space setting.

\begin{MainTheorem}\label{Global solution supercritical Besov data}
    Let $2<q<\infty$ and $-1+\frac{2}{q}<s<\min \left(-1+\frac{3}{q},0 \right).$ Suppose that $u_0 \in \Bes{s}{q}{\infty}(\R^3)$ is divergence-free. Then there exist constants $\max(q,3)< p<\infty$ and $0<\delta<1-\frac{3}{p}$ (both depending on $q$ and $s$) and a suitable\footnote{As in Definition \ref{suitablesolngen}.} weak solution $(u,\pi)$ to the Navier-Stokes equations on $\R^3\times(0,\infty)$ which may be decomposed as $$u=m+w \in L_{loc}^\infty(0,\infty;L^\infty(\R^3)\cap L^p(\R^3)) + L^\infty(0,\infty;L^2(\R^3))\cap L^2(0,\infty;\dot{H}^1(\R^3)),$$ corresponding to an initial data splitting $$u_0  = m_0 + w_0 \in \Bes{-1+\frac{3}{p}+\delta}{p}{p}\cap \Bes{s}{q}{\infty} + L^2\cap \Bes{s}{q}{\infty} $$ such that:
    \begin{enumerate}
        \item[(a)] $m$ is a mild solution of (\ref{Navier-Stokes equations}) on $\R^3\times(0,\infty)$ and satisfies
    \begin{equation}
     \esssup_{0<t<\infty} t^{\frac{1}{2}}\norm{m(\cdot,t)}{L^\infty}<\infty,\qquad \esssup_{0<t<\infty} t^{\frac{1}{2}\left(1-\frac{3}{p} \right)}\norm{m(\cdot,t)}{L^p}<\infty.  \nonumber
    \end{equation}
    and \begin{equation}
        \esssup_{0<t<\infty} t^{\frac{1}{2}\left(1-\frac{3}{p}-\delta \right)}\norm{m(\cdot,t)}{L^p}<\infty.  \nonumber
    \end{equation}

    \item[(b)]w satisfies: 
    \begin{enumerate}
        \item[(i)] $\lim\limits_{t\to0^+}\norm{w(\cdot,t)-w_0}{L^2}=0$.
\item[(ii)] for all $\phi \in C^\infty_c(\R^3),$ the map $t \mapsto \int_{\R^3}w(x,t)\cdot\phi(x)\;dx$ is continuous on $[0,\infty).$ 
\item[(iii)]w satisfies the (perturbed) energy inequality: for almost every $t_1 \in \left[0,\infty\right)$ (and including $t_1=0$) and for all $t_2 \in \left[ t_1, \infty\right)$, 
    \begin{equation}
        \norm{w(\cdot,t_2)}{L^2_x}^2 +2\int_{t_1}^{t_2} \norm{\nabla w(\cdot,s)}{L^2_x}^2ds \leq \norm{w(\cdot,t_1)}{L^2_x}^2 + 2\int_{t_1}^{t_2} \int_{\R^3} m\cdot(w\cdot\nabla)w\;dxds. \nonumber
    \end{equation}
\item[(iv)] $w$ is suitable in the sense of Definition \ref{suitablesolngen} for the $m$-perturbed Navier-Stokes equations on $\R^3\times (0,\infty).$
    \end{enumerate}      
    \end{enumerate}
\end{MainTheorem}
Utilising the Calderón-inspired splitting technique which we used to prove Theorem A, we also address \textbf{(Q2)} and prove our second main Theorem, which assumes a sequential Type-I blow-up assumption in the Besov-space setting.
\begin{MainTheorem} \label{Type1 singular set}
    Let $3<q<\infty$, $-1+\frac{3}{q}\leq s<0$ and $M\geq1$. Let $u: \R^3\times(-1,\infty) \rightarrow \R^3$ be a suitable Leray-Hopf weak solution to the Navier-Stokes equations, where we assume:
\begin{enumerate}
\item[(a)] $u$ first loses smoothness at time $0$;
\item[(b)] there exists an increasing sequence $t_n \in (-1,0)$ with $t_n\uparrow 0$ such that  
 \begin{align} 
 \sup_{n\in\N}\left[ (-t_n)^{\frac{1}{2}\left(s+1-\frac{3}{q}\right)}\norm{u(\cdot,t_n)}{\Bes{s}{q}{\infty}(\R^3)} \right] \leq M. \label{BesovtimeslicetypeI}
 \end{align}
\end{enumerate}  
Then there exists a constant $C_{q,s}\geq1$ such that the number of elements $N$ of the singular set at time $0$ is bounded above by
\begin{equation}
   N \leq C_{q,s}M^{6q-9}.
\end{equation}
\end{MainTheorem}
\pagebreak
\subsection{Comparison with previous literature} 
\subsubsection{Theorem \ref{Global solution supercritical Besov data} and (\textbf{Q1})} Let us first review the literature for \textbf{(Q1)} in Lebesgue spaces $L^p$ for $p\geq 2$:
We have already mentioned Leray's result from supercritical\footnote{Subcritical, critical and supercritical means in the sense of Section \ref{criticaldefn}.} $L^2(\R^3)$ data. For small enough critical $L^3(\R^3)$ initial data, it has also been shown independently by Giga and Kato that there exists a global-in-time smooth solution to (\ref{Navier-Stokes equations}); see \cite{Kato84, Giga}. We also mention the existence of \textit{local}-in-time solutions to (\ref{Navier-Stokes equations}) for subcritical initial data in $L^p(\R^3)$ for $p>3$ (see \cite{Leray, Kato64, FJR72}). Later, Calderón \cite{Calderón} filled the remaining supercritical gap between these results, showing that supercritical initial data in $L^p(\R^3)$ for $2<p<3$ gives rise to global-in-time weak solutions to (\ref{Navier-Stokes equations}). This was achieved by splitting the initial data $u_0=w_0+m_0$ into an energy piece and a small critical piece: $w_0\in L^2(\R^3)$ and $m_0 \in L^3(\R^3)$. Then the Giga-Kato construction is used to construct a global-in-time mild solution $m$ (up to a certain time) corresponding to $m_0$, and $m$ lives in a critical space. Next, one wishes to apply arguments similar to Leray to find a weak solution $w$ to the $m$-perturbed Navier-Stokes equations:
\begin{equation}\label{Perturbed Navier-Stokes equations} \tag{m-NSE}
    \partial_t w -\Delta w +(m \cdot \nabla)w + (w\cdot \nabla)w + (w \cdot \nabla)m +\nabla \tilde{\pi} =0,\qquad   \text{div}(w) = 0.
\end{equation}
Usually, the problem with weakly solving these equations essentially comes from trying to control the energy of $w$, through the integral 
\begin{equation}
    \int_0^t\int_{\R^3}(w \cdot\nabla) w \cdot m \;dxds. \label{Calderon integral}\end{equation}
In Calderón's setting, $m$ lives in a certain critical space which allowed this integral to be controlled\footnote{This is made clear in lines (\ref{Calderoncontrol1})-(\ref{Calderoncontrol2}).} giving the energy bounds for $w$ that we require. Then $u:=m+w$ weakly solves (\ref{Navier-Stokes equations}) with pressure $\pi:=\tilde{\pi}+\pi_m$ (where $\pi_m$ is pressure associated to $m$). Calderón used a priori bounds to show that this solution could be indefinitely extended in time.

Calderón's paper has gone on to inspire many notions of global weak solutions. Lemarié-Rieusset (see \cite{Lemarie}) extended Calderón's result to show that there exist global-in-time weak solutions corresponding to initial data in $L^2_{uloc}(\R^3)$ with certain decay properties. Seregin and Šverák \cite{largeL3data} showed that for \textit{any} size initial data in $L^3(\R^3)$, there exists a global-in-time weak $L^3$-solution. Next, Barker, Seregin \& Šverák \cite{L3infweaksolns} used an initial data splitting reminiscent of Calderón's splitting (but in the Lorentz space setting) and showed that there exist global-in-time weak $L^{3,\infty}$-solutions for initial data in $L^{3,\infty}(\R^3)$ of any size. Finally, Albritton and Barker \cite{AlbrBarker} used a Calderón-type splitting of initial data in the Besov space setting to prove that there exist global-in-time weak Besov-solutions corresponding to initial data of any size in critical Besov spaces $\Bes{-1+\frac{3}{p}}{p}{\infty}(\R^3)$, $p>3$. The splitting technique from Calderón's paper has been influential in other areas of the theory of the Navier-Stokes equations. For example, it has inspired similar splittings in \cite{asymptoticLp,bradshawasymptotic} for asymptotic stability, \cite{eventualregTsai} for eventual regularity and \cite{Barkeruniq17, PierreWeakStrong} for weak-strong uniqueness. It has also inspired works outside of the Navier-Stokes equations, \cite{Calderoninsemilinear, CalderoninMHD, Calderoningeostrophic} for example.   

Theorem \ref{Global solution supercritical Besov data} fills a similar \textit{supercritical} gap to Calderón's result, but in the setting of Besov spaces. This is the gap that we fill: first, Leray gives a global solution corresponding to initial data in $\Bes{0}{2}{2}(\R^3)=L^2(\R^3)$. For small enough subcritical and critical Besov initial data in $\Bes{s}{p}{\infty}(\R^3)$ for $p\geq 3$ and $-1+\frac{3}{p}\leq s<0$, the mild solution theory is known: we present it again in this paper (Theorem \ref{Besovmildsoln}) for completeness.  Using an initial data splitting (based on the one seen in \cite{AlbrBarker}) corresponding to Calderón's initial data splitting, we prove Theorem \ref{Global solution supercritical Besov data} in Section 4, filling the supercritical gap. Theorem \ref{Global solution supercritical Besov data} actually extends Calderón's result, via the embedding $L^p(\R^3) \hookrightarrow \Bes{0}{p}{\infty}(\R^3) \hookrightarrow \Bes{-3\left(\frac{1}{p}-\frac{1}{q}\right)}{q}{\infty}(\R^3)$ (and choosing $q>p$ appropriately).

The trick to prove this theorem again partially lies in controlling the integral (\ref{Calderon integral}). If we try splitting $u_0 = m_0+w_0$ purely into an $L^2$ piece and a critical piece exactly like Calderón, our situation is not as nice in the Besov space setting: where we would need to gain control of (\ref{Calderon integral}), we would end up with a logarithmic singularity\footnote{To see this observe the following: in line (\ref{subcritrequirement}), setting $\delta= 0$ corresponds to a critical (Kato) space. The resulting integrand is $s^{-1}$. Related issues are also observed in \cite{AlbrBarker, Barkeruniq17}}. Hence, we require $m$ to live in a subcritical space. On the other hand, if we split $u_0=m_0+w_0$ purely into an $L^2$ piece and a subcritical piece, then the mild solution $m$ corresponding to $m_0$ would only be defined locally-in-time. 

We combine these two ideas, using the persistency property of the Besov space initial data splitting (line (\ref{persistency1})) and the observation that one may propagate the subcriticality properties of the initial data to a mild solution through Picard iterates (see Lemma \ref{contractionLemma} (ii)). Indeed, the persistency property allows one to perform a splitting of the initial data into $L^2$ and subcritical pieces, then use Besov interpolation to ensure that a critical Besov norm of $m_0$ is sufficiently small (see line (\ref{critpart})) to generate a global-in-time mild solution living in a critical space. Propagating the subcriticality of the initial data will ensure that this mild solution also lives in a subcritical space (see line (\ref{subcritkatoestim})).  

Finally, let us mention that the case $q=\infty$ for Theorem A appears out of reach; the paper \cite{IllposednessBesov} of Bourgain and Pavlović shows ill-posedness\footnote{In the sense that a “norm inflation” occurs in finite time.} for the Cauchy problem for initial data in $\Bes{-1}{\infty}{\infty}$. A Leray-Hopf solution with a discontinuity at $t=0$ with respect to the $\Bes{-1}{\infty}{\infty}$ metric was also obtained by Cheskidov and Shvydkoy in \cite{Shvydkoyillposed}.

\subsubsection{Theorem \ref{Type1 singular set} and (\textbf{Q2})} Our second question seeks to improve our understanding of potential singular solutions of (\ref{Navier-Stokes equations}). Inspired by the Navier Stokes rescaling (\ref{NSrescale}), Leray \cite{Leray} first posed the blow-up ansatz of backward self-similar solutions. These are solutions of the form 
\begin{equation}
    v(x,t) = \frac{1}{\sqrt{\lambda (T-t)}}V\left(\frac{x}{\sqrt{\lambda (T-t)}}\right)
\end{equation}
which are invariant with respect to the Navier-Stokes rescaling  (\ref{NSrescale}). However, backward self-similar solutions were ruled out in \cite{selfsimNeustupa} and \cite{selfsimTsai}. Other blow-up ansätze  have been investigated, such as discretely backwards self-similar (DSS) solutions \cite{Chae2016RemovingDS}. All of the blow-up ansätze mentioned here satisfy more general Type-I bounds (for example, see \cite{Chae2016RemovingDS} for backwards DSS solutions), making this a natural setting to study potential singularity formation. Furthermore, we have already drawn attention to the paper \cite{computerblowup}; in this paper, numerical computations suggest potentially singular behaviour for the \textit{axisymmetric} Navier-Stokes equations and a `nearly' self-similar blow-up profile was observed\footnote{We note that Type-I blow-up has been ruled out in the axisymmetric setting \cite{notypeIaxis}. However, it has not yet been ruled out for the general Navier-Stokes equations.} (see Sections 3.5-3.6 in particular). 

A natural avenue of investigation is to study the \textit{structure} of the putative singular set. For instance, the parabolic Hausdorff measure of the singular set is considered in \cite{CKN82, logimprovement, LeiRen, Barkersupercritpress, RobinsonBox} to name a few\footnote{The box-counting dimension of the singular set was also considered in \cite{RobinsonBox}.}.  
The structure of the singular set under the Type-I ansatz has been previously considered in the literature as well; the paper \cite{ChoeWolfYang} of Choe, Wolf and Yang demonstrates that a suitable finite-energy solution $u:\R^3\times [0,T^*] \to \R^3$ which loses smoothness at time $T^*$ and satisfies the Type-I bound \begin{equation}
    \norm{u}{L^\infty_t((0,T^*);L^{3,\infty}(\R^3))} \leq M \qquad \text{(where} \; M\geq 1 \; \text{is sufficiently large)} \label{weakl3typeI}
\end{equation}
may only have a finite number of singular points at $T^*$, qualitatively depending on $M$. This result was then extended by Seregin \cite{sereginnote} to hold locally for suitable weak solutions. We note that these are qualitative results. In \cite{spatialconc}, Barker and Prange quantified this result, bounding the number of singular points at $T^*$ under the Type-I assumption (\ref{weakl3typeI}) above by $\exp(\exp(M^{1024})).$ These results all rely on unique continuation and  backward uniqueness results as in \cite{ESS03}. Later, Barker \cite{Barkersing21} instead used elementary arguments to improve the previous bound: for a Leray-Hopf solution $u:(-1,\infty)\to \R^3$ that first blows-up at time $0$ satisfying the  more general time-slice Type-I bound 
\begin{equation}
    \sup_n\norm{u(\cdot, t_n)}{L^{3,\infty}(\R^3)} \leq M \qquad \text{(for some increasing sequence} \; t_n \uparrow 0,\; M\geq 1)\label{weakl3timeslicetypeI}
\end{equation}
the number of singular points at $0$ is bounded above by $C_{univ}M^{20}.$ 
In this paper, we show that under a Besov space Type-I blow-up ansatz (\ref{BesovtimeslicetypeI}), the number of points of the singular set is quantitatively bounded above in terms of the Type-I bound. This setting for the Type-I assumption is less regular than what has been considered in the literature so far to the author's knowledge. 

The proof of Theorem \ref{Type1 singular set} is based on the elementary arguments from the proof of Theorem 2 of \cite{Barkersing21}. However, we utilise the Calderón-type splitting for Besov spaces to achieve our quantitative estimates. Indeed, this splitting allows us to move from the difficult-to-work-with Besov setting, and instead deal with a nicely behaved mild solution $m$, and a Leray-type solution $w$ to the $m$-perturbed Navier-Stokes equation. This allows us to utilise mild solution estimates for $m$ and Leray-type energy estimates $w$ in conjunction with an $\epsilon$-regularity criterion to bound the number of singular points.

\section{Preliminaries}
\subsection{General notation}
Throughout this paper, we adopt Einstein summation convention. We also adopt the convention where constants may change from line-to-line unless made otherwise clear. \\We often write $X \lesssim_{a,b,c...} Y$ if there exists a universal constant $C>0$ depending on parameters $a,b,c...$ such that $X \leq C_{a,b,c...}Y.$ \\The notation $X \sim_{a,b,c...} Y$ means that both $X \lesssim_{a,b,c...} Y$ and $X \gtrsim _{a,b,c...} Y.$ 
\\\\In the standard basis $(\textbf{e}_i)$ for $\R^3$, the tensor product of two vectors $a=a_i\textbf{e}_i,$ $b=b_i\textbf{e}_i$ is the matrix $a\otimes b$ with entries
$$(a\otimes b)_{ij} := a_ib_j.$$
\\For $x_0\in \R^n$, $r>0$ we denote the ball
$$B_r(x_0):= \left\lbrace x\in \R^n : |x-x_0|<r \right\rbrace.$$
We often write $B_r := B_r(0)$.
For $z_0= (x_0,t_0) \in \R^n \times \R$, $r>0$ we denote the parabolic cylinder
$$Q_r(z_0) := B_r(x_0) \times (t_0-r^2,t_0).$$
We often write $Q_r := Q_r(0,0)$.
\\\\For any Banach space $(X, \norm{\cdot}{X})$, $a<b$, $p\in[1,\infty)$ we denote the mixed space $L^p(a,b;X)$ to be the Banach space of strongly measurable functions $f:(a,b)\to X$ such that
$$\norm{f}{L^p(a,b;X)}:= \left( \int_a^b \norm{f(t)}{X}^p\right)^\frac{1}{p} < \infty$$
with the usual modifications for $p=\infty$ or for locally integrable spaces. When it is clear, we will often denote $$L^p_tL^q_x := L^p(0,T;L^q(\R^n)) \qquad \text{and} \qquad L^p_{t,x}:= L^p_tL^p_x.$$

\subsection{Besov and Kato spaces}
$\Bes{s}{p}{q}(\R^d)$ denotes the Besov space with regularity index $s$ with integrability $p$ and frequency summation $q$ over $\R^d$; see the Appendix \ref{Besovspaces}. We have the following embedding of Besov spaces; this is Proposition 2.20 in \cite{Bahouri}:
\begin{Proposition} \label{Besov embeddings} Let $s \in(-\infty,\frac{d}{p})$, $1\leq p_1 \leq p_2 \leq \infty$ and $1\leq r_1 \leq r_2 \leq \infty$. Then we have the continuous embedding $\Bes{s}{p_1}{r_1}(\R^d) \hookrightarrow \Bes{s-d\left(\frac{1}{p_1}-\frac{1}{p_2} \right)}{p_2}{r_2}(\R^d)$:
\begin{align*}
\norm{f}{\Bes{s-d\left(1/p_1-1/p_2\right)}{p_2}{r_2}} \lesssim_{d,p_1,p_2,r_1,r_2}\norm{f}{\Bes{s}{p_1}{r_1}}.
\end{align*}
\end{Proposition}
We also have the following interpolation theorem for Besov spaces; this is Theorem 2.80 in \cite{Bahouri}:
\begin{Theorem} \label{Besovinterpolation}
    There exists a universal constant $C>0$ such that, if $s_1<s_2$, $\theta \in(0,1)$ and $1\leq p,r \leq \infty$, then 
    \begin{align}
    \norm{f}{\Bes{\theta s_1+(1-\theta)s_2}{p}{r}} &\leq \norm{f}{\Bes{s_1}{p}{r}}^\theta\norm{f}{\Bes{s_2}{p}{r}}^{1-\theta},
    \\    \norm{f}{\Bes{\theta s_1+(1-\theta)s_2}{p}{1}} &\lesssim \frac{1}{s_2-s_1}\left(\frac{1}{\theta} - \frac{1}{1-\theta}\right)\norm{f}{\Bes{s_1}{p}{\infty}}^\theta\norm{f}{\Bes{s_2}{p}{\infty}}^{1-\theta}.
    \end{align}
\end{Theorem}
We will make use of the heat flow characterisation of Besov spaces with negative regularity index; this is a variant of Theorem 2.34 in \cite{Bahouri}:
\begin{Proposition} \label{heat characterisation}
For all $-\infty<s<0$, and $f \in \Bes{s}{p}{\infty}(\R^d)$,
\begin{equation}
\norm{f}{\Bes{s}{p}{\infty}} \sim_s \;\esssup_{t>0}t^{-\frac{s}{2}}\norm{e^{t\Delta}f}{L^p}. \nonumber
\end{equation}
\end{Proposition}
Here, $e^{t\Delta}$ is the heat operator $e^{t\Delta}f:= \Gamma(\cdot,t)\ast f $ for $t>0$ where $\Gamma:\R^d\times(0,\infty)\to\R$ is the heat kernel. This characterisation motivates the definition of the following Kato spaces:

\begin{Definition} \label{Katospace} For $0<T\leq \infty$, $3\leq p\leq \infty$ and $0\leq\delta\leq 1-\frac{3}{p}$, the Kato space $K^{p,\delta}(T)$ is defined by
\begin{equation}
K^{p,\delta}(T) = \left\lbrace u\in L^\infty_{loc}(0,T;L^p): \norm{u}{K^{p,\delta}(T)}:= \esssup_{0<t<T} t^{\frac{1}{2}\left(1-\frac{3}{p}-\delta \right)}\norm{u(\cdot,t)}{L^p}<\infty \right\rbrace . \nonumber
\end{equation} 
For ease of notation, we set $K^p(T):=K^{p,0}(T)$.
\end{Definition}
Note that, for $\delta \in \left[0,1-\frac{3}{p}\right)$, the norm $\norm{e^{t\Delta}f}{K^{p,\delta}(\infty)}$ is equivalent to $\norm{f}{\Bes{-1+\frac{3}{p}+\delta}{p}{\infty}}$ by Proposition \ref{heat characterisation}. This is a useful observation when constructing mild solutions.
\begin{Remark}\label{kato is serrin}
For $p\in (3,\infty]$ and $0<T<\infty$, we have $K^{p,\delta}(T) \hookrightarrow L^q(0,T;L^p(\R^3))$, where $q$ is defined by the relation $\frac{2}{q}+\frac{3}{p}=\theta > 1-\delta$. 
\end{Remark}

\subsection{General definitions for the Navier-Stokes equations}
\subsubsection{Critical quantities for the Navier-Stokes equations} \label{criticaldefn}
Let us recall what it means for a scalar quantity $\mathcal{X}(u,\pi,u_0)$ to be critical with respect to the Navier-Stokes rescaling. For all $\lambda>0$, suppose that $\mathcal{X}$ satisfies $\mathcal{X}(u_\lambda,\pi_\lambda,u_{0,\lambda}) = \lambda^\alpha\mathcal{X}(u,\pi,u_0)$ on the whole space-time domain, where $$u_\lambda(x,t):=\lambda u(\lambda x, \lambda^2 t) , \pi_\lambda:= \lambda^2 \pi(\lambda x, \lambda^2 t), \quad \text{and} \quad u_{0,\lambda}:= \lambda u_0(\lambda x),$$ then:
\begin{itemize}
    \item if $\alpha>0$, then $\mathcal{X}$ is \textbf{subcritical} with respect with the Navier-Stokes rescaling.
    \item if $\alpha=0$, then $\mathcal{X}$ is \textbf{critical} with respect with the Navier-Stokes rescaling.
    \item if $\alpha<0$, then $\mathcal{X}$ is \textbf{supercritical} with respect with the Navier-Stokes rescaling.
\end{itemize}
In this paper, we are mainly interested in initial data classes. In particular, the Besov spaces $\Bes{s}{p}{\infty}(\R^3)$ with $p\geq 3$ and negative regularity index $s$ are subcritical if $-1+\frac{3}{p}<s<0$, critical if $s=-1+\frac{3}{p}$ or supercritical if $s<-1+\frac{3}{p}$.

\subsubsection{Leray-Hopf weak solutions}
A function $u$ is a \textbf{Leray-Hopf weak solution} to the Navier-Stokes equations on $\R^3\times(0,\infty)$ corresponding to divergence-free (in the sense of distributions) initial data $u_0 \in L^2$ if:
\begin{itemize}
    \item $u$ satisfies the Navier-Stokes equations in the sense of distributions.

    \item $u \in L^\infty((0,\infty);L^2(\R^3)) \cap L_{\textup{loc}}^2([0,\infty);\dot{H}^1(\R^3)). $

    \item  $u$ attains the initial value in the $L^2$ sense:
\begin{equation}
\lim_{t\to0^+}\norm{u(t)-u_0}{L^2} = 0. \nonumber
\end{equation}

    \item $u$ satisfies the (strong) energy inequality:
\begin{equation} \label{energy inequality}
\norm{u(\cdot,t)}{L^2}^2+2\int_s^t\norm{\nabla u(\cdot,\tau)}{L^2}^2 d\tau \leq \norm{u(\cdot,s)}{L^2}^2 \nonumber
\end{equation}
for almost every $s>0$ and for $s=0$, and for every $t>s$. 

  \end{itemize} 

\subsubsection{The (putative) singular set}\label{singularsection}
Let $v:\R^3\times (0,\infty)\to\R^3$ be a Leray-Hopf solution to (\ref{Navier-Stokes equations}). It is known that, if a solution first looses smoothness at a time $T$, then there exists a singular point in the following sense:
\\ A space-time point $(x_0,T)\in \R^3\times(0,\infty)$ is said to be a \textbf{singular point} of $v$ if $v\not\in L^\infty(Q_r(x_0,T))$ for all $r>0$ sufficiently small. Else, $(x_0,T)$ is called a \textbf{regular point} of $v$.
\\ We let $\sigma(T)$ denote the \textbf{singular set at time $\boldsymbol{T}$} the set of all singular points of $v$ that occur at time $T$. 
\\ Finally, we denote the $\mathbf{\varrho}$\textbf{-isolated singular set at time} $\boldsymbol{T}$ as 
\begin{equation}
\sigma_\varrho(T) := \Big\lbrace x \in \sigma(T) \ \Big\vert \ B_\varrho(x)\cap B_\varrho(y) = \emptyset \quad \forall y \in \sigma(T)\setminus\lbrace x \rbrace \Big\rbrace, \nonumber
\end{equation}
the set of all singular points of $v$ at time $T$ that are a distance of at least $\varrho$ from any other singular point of $v$ at time $T$.

\subsection{Mild solution notation}
 On the whole space, the \textbf{Leray projector} $\LP$ is defined as the operator\footnote{or as a Fourier multiplier with symbol $1-\frac{\xi \otimes \xi}{|\xi|^2}.$} $\LP := \text{Id} + \nabla (-\Delta)^{-1}\text{div}$. This maps functions of $L^2$ to the closed set of (weakly) divergence-free $L^2$ functions, and is a continuous mapping on homogeneous Besov spaces (see Proposition 2.30 of \cite{Bahouri}.) 
 Define 
\begin{align}
\left[ B(a,b) \right] (x,t) &= -\frac{1}{2}\int_0^te^{(t-\tau)\Delta}\LP \text{div} \left( a\otimes b +b\otimes a \right) d\tau
\end{align}
with $i$-th component
\begin{align}
\left[ B(a,b) \right]_i (x,t)  &= \int_0^t\int_{\R^3}K_{ijl}(x-y,t-\tau)a_j(y,\tau) b_l(y,\tau)  dyd\tau.
\end{align}
The operator $-e^{t\Delta}\LP\text{div}$ is a convolution operator with the Oseen tensor $K=(K_{ijl})$ as its kernel. See, for example, Chapter 11 of \cite{Lemarie}. It is known (for example, from pp.234-235 of \cite{SolonnikovRussian}) that the Oseen kernel $K_{ijl}$ satisfies the bound
\begin{equation}
\left\vert \partial_t^a\nabla_x^bK_{ijl}(x,t)\right\vert \leq \frac{C_{\alpha,\beta}}{\left( |x|^2+t\right)^{2+\alpha+\frac{\beta}{2}}}.
\end{equation}
In particular, we will use that
\begin{equation} \label{Oseen Lp bnd}
\norm{K(\cdot,t)}{L^p} \leq \frac{C}{t^{\frac{1}{2}\left(4-\frac{3}{p}\right)}}.
\end{equation}
A \textbf{mild solution} of the Navier-Stokes equation is a fixed point of the map
\begin{equation}
u \mapsto e^{t\Delta}u_0+B(u,u).
\end{equation}
We now recall how this lemma may be used to construct mild solutions to (\ref{Navier-Stokes equations}) from initial data in critical and subcritical Besov spaces \cite{CannoneBesovmild, PlanchonBesovmild}. This will establish notation for when we later propagate additional information from the initial data to the whole mild solution by using Lemma \ref{contractionLemma} (ii).
\begin{Theorem} \label{Besovmildsoln} Let $3<p<\infty$ and let $0\leq \delta<1-\frac{3}{p}$. Suppose that $v_0 \in \Bes{-1+\frac{3}{p}+\delta}{p}{\infty}(\R^3)$ is divergence-free. Then there exist constants $C_{p,\delta}^{(1)},C_{p,\delta}^{(2)}>0$ depending on $p$ and $\delta$ so that if
\begin{equation} \label{Besovinfty}
T^\frac{\delta}{2}\norm{v_0}{\Bes{-\left(1-\frac{3}{p}-\delta \right)}{p}{\infty}} \leq C_{p,\delta}^{(1)}, 
\end{equation}
then there exists a mild solution $v \in K^{p,\delta}(T)\cap K^{\infty,\delta}(T)$ to the Navier-Stokes equations satisfying
 
\begin{equation}
\norm{v}{K^{p,\delta}(T)}+\norm{v}{K^{\infty,\delta}(T)} \leq 2\left(\norm{e^{t\Delta}v_0}{K^{p,\delta}(T)} + \norm{e^{t\Delta}v_0}{K^{\infty,\delta}(T)}\right)\leq C_{p,\delta}^{(2)}T^{-\frac{\delta}{2}}.
\end{equation} 
\begin{proof}
    See Appendix \ref{proofofmildtheorem}.
\end{proof}
\end{Theorem}
\begin{Remark}
    Looking at the proof in Appendix \ref{proofofmildtheorem}, it is easy to see that when $\delta=0$, we may take $T=\infty$ provided that $\norm{v_0}{\Bes{-\left(1-\frac{3}{p} \right)}{p}{\infty}} \leq C^{(1)}_{p,0}$ to find $\norm{v}{K^{p,0}(\infty)}+\norm{v}{K^{\infty,0}(\infty)}\leq
C_{p,0}^{(2)}.$
\end{Remark}

\subsection{The $\boldsymbol{v}$-perturbed Navier-Stokes equations}
It will often be useful to consider the $v$-perturbed Navier-Stokes equations: given a divergence-free function $v$, a solution $w$ satisfies the $v$-perturbed Navier-Stokes equations if it satisfies: 
\begin{equation} \label{v-pert NSE}
    \partial_t w -\Delta w +(v \cdot \nabla)w + (w\cdot \nabla)w + (w \cdot \nabla)v +\nabla \tilde{\pi} =0,\qquad   \text{div}( w )= 0.
\end{equation}
in the sense of distributions. Note that if $(v,\pi_v)$ satisfies (\ref{Navier-Stokes equations}) and $(w,\tilde{\pi})$ satisfies (\ref{v-pert NSE}), then $(v+w,\pi_v+\tilde{\pi})$ satisfies (\ref{Navier-Stokes equations}).
There is a generalised notion of a local weak suitable solution for the $v$-perturbed Navier-Stokes equations:
\begin{Definition} \label{suitablesolngen} Let $\Omega \subseteq \R^3$ and suppose that $v\in L^p_tL^q_x(\Omega)$ is divergence-free with $\frac{2}{p}+\frac{3}{q}<1$ and $q>3$. A pair $(w,\tilde{\pi})$ is a \textbf{suitable weak solution} to the $v$-perturbed Navier-Stokes equations on the domain $\Omega \times (a,b)$ if they satisfy them in the distributional sense and $$w \in L^{\infty}((a,b);L^2(\Omega)) \cap L^2((a,b);\dot{H}^1(\Omega)).$$ Moreover, $$\tilde{\pi}\in L^{3/2}(\Omega \times (a,b))$$ with $$
-\Delta\tilde{\pi} = \partial_i\partial_j(w_iw_j)+ 2\partial_i\partial_j(v_iw_j)
$$ for almost every $t\in(a,b)$. Finally, we assume that $w$ satisfies the generalised local energy inequality
\begin{multline*} 
\int_\Omega |w(x,t)|^2\phi(t)\,dx + 2\int_a^t\int_\Omega|\nabla w|^2\phi \;dxds \\ \leq \int_a^t\int_\Omega |w|^2(\partial_t \phi + \Delta \phi)\,dxds + \int_a^t\int_\Omega (|w|^2(w+v)+2\tilde{\pi}w)\cdot\nabla\phi \;dxds \\ + 2\int_a^t\int_\Omega (w\cdot\nabla\phi)(v\cdot w) + ((w\cdot \nabla) w)\cdot v\phi \;dxds
\end{multline*}
for almost every $t\in(a,b)$ and all non-negative $\phi \in C_0^\infty(\Omega \times (a,\infty))$.
\end{Definition}
Note that this coincides with the usual definition for suitable solutions to the Navier-Stokes equations when $v=0$. 
\par The following result will be useful in proving Theorem \ref{Besovenergy2} when we decompose a suitable weak solution to the Navier-Stokes equations into a piece that solves the Navier-Stokes equations and a piece that solves the perturbed equation. In particular, note that Remark \ref{kato is serrin} (with $\theta = 1-\frac{\delta}{2}<1$) implies that the condition $v\in K^{p,\delta}(T)$ will allow us to apply the following theorem:
\begin{Theorem}[\cite{peturbedCKN}] \label{peturbedCKN} Let $v\in L^p_tL^q_x(Q_\varrho)$ be divergence-free with $\frac{2}{p}+\frac{3}{q}<1$ and $q>3$. Let $(w,\tilde{\pi})$ be a suitable weak solution to the $v$-perturbed Navier-Stokes equations on $Q_\varrho$. Then there exists $\epsilon_{p,q} \in(0,1)$ and a universal constant $C>0$ such that the condition 
\begin{equation}
\frac{1}{\varrho^2}\int_{Q_\varrho}|w|^3+|\tilde{\pi}|^\frac{3}{2}dxdt \leq \epsilon_{p,q}
\end{equation} 
implies that $\norm{w}{L^\infty (Q_{\varrho / 2})} \leq C\varrho^{-1}$.
\end{Theorem}
\begin{proof}[Proof of this version] This result is essentially Theorem 4.1 of Li, Miao and Zheng's paper \cite{peturbedCKN}, except in \cite{peturbedCKN}, $\epsilon_{p,q}$ is replaced by $\epsilon(p,q, \norm{v}{L^p_tL^q_x(Q_\varrho)})$. To obtain the version of the Theorem presented above, we employ a simple rescaling argument to remove the dependency of $\epsilon$ on $\norm{v}{L^p_tL^q_x}$. 
\\ Indeed, taking the Theorem 4.1 of \cite{peturbedCKN} and performing the Navier-Stokes rescaling gives
$$\norm{v_\lambda}{L^p_tL^q_x(Q_{\varrho/\lambda})}=\lambda^{1-\frac{2}{p}-\frac{3}{q}}\norm{v}{L^p_tL^q_x(Q_\varrho)}.$$
Then, using that $\frac{2}{p}+\frac{3}{q}<1$, we set $\lambda := \norm{v}{L^p_tL^q_x(Q_\varrho)}^{-\frac{1}{\left(1-\frac{2}{p}-\frac{3}{q}\right)}}$ to ensure that $\norm{v_\lambda}{L^p_tL^q_x(Q_{\varrho/\lambda})} =1$.
Then, we apply Theorem 4.1 of \cite{peturbedCKN}: there exists $\epsilon_{p,q}:=\epsilon(p,q,1)>0$ and\footnote{We note that, without loss of generality, we may assume that $\epsilon_{p,q}<1.$} a universal constant $C>0$ such that the condition 
\begin{equation}
\frac{1}{\varrho^2}\int_{Q_{\varrho}}|w|^3+|\tilde{\pi}|^\frac{3}{2}dxdt=\frac{\lambda^2}{\varrho^2}\int_{Q_{\varrho/\lambda}}|w_\lambda|^3+|\tilde{\pi}_\lambda|^\frac{3}{2}dxdt \leq \epsilon_{p,q} \label{rescaledpeturbedcond}
\end{equation} 
implies that $\norm{w_\lambda}{L^\infty (Q_{\varrho / 2\lambda})} \leq C(\frac{\varrho}{\lambda})^{-1}$. 
This entails that $\norm{w}{L^\infty (Q_{\varrho / 2})} \leq C\varrho^{-1}$. 
\end{proof}
Note that this Theorem coincides with the usual Caffarelli-Kohn-Nirenberg Theorem when $v=0$.
\section{Calderón-type splitting of Besov Space initial data}
The method of this paper is inspired by Calderón's seminal paper \cite{Calderón}.
To show the global existence of weak solutions with initial data in $L^p$ for $2<p<3$, Calderón used an elementary Lebesgue space splitting: if and $1 < a < p < b < \infty$ and $N>0$ and $u_0 \in L^p(\R^3)$ is divergence-free, then we may decompose $u_0=u_0^{>N}+u_0^{\leq N}$, where $u_0^{>N}$ and $u_0^{\leq N}$ are (weakly) divergence-free and satisfy
\begin{equation*}
\|u_0^{\leq N}\|^b_{L^b} \lesssim_bN^{\left(b-p\right)} \|u_0\|^p_{L^p} 
\end{equation*} 
and
\begin{equation*}
 \|u_0^{>N}\|^a_{L^a} \lesssim_aN^{-\left(p-a\right)} \|u_0\|^p_{L^p} .
\end{equation*} 
Calderón used $a=2$ and $b=3$ to split the initial data into two pieces. The $L^3$ ($b=3$) part along with a choice of $N$ small enough corresponds to a piece of critical initial data with sufficiently small norm. A global-in-time mild solution may be constructed corresponding to this piece according to mild theory. For the $L^2$ ($a=2$) part, one has an associated global-in-time weak solution with an energy estimate, controlled in terms of the size of the original initial data in $L^p$. 

A similar splitting is available for Besov spaces of the form $\Bes{s}{p}{\infty}$. First, we state a slightly more general splitting result for Besov spaces in $\R^d$, as seen in Section 5 of \cite{AlbrBarker}.  

\begin{Proposition} \label{genBesovsplit}
Let $1\leq p_1 <q< p_2\leq \infty$, $s,s_1 \in \R$, $d\in\N$.
Let $l_1$ be the unique line through points $(s,\frac{1}{q})$ and $(s_1,\frac{1}{p_1})$, let $l_2$ be the unique line through $(s,\frac{1}{q})$ with gradient $\frac{1}{d}$ and let $l_3$ be the horizontal line through the origin. Assuming $l_1$ and $l_2$ are distinct, let $D$ be the interior of the compact region bounded by these three lines. 
Then let $s_2\in\R$ be such that $(s_2,\frac{1}{p_2})\in D$, and let $\bar{s} \in\R$ be the unique value such that $(\bar{s},\frac{1}{q}) \in l_4$ where $l_4$ is the line connecting $(s_1,\frac{1}{p_1})$ and $(s_2,\frac{1}{p_2})$, and let $\tilde{s}:=s_2+\frac{d}{q}-\frac{d}{p_2}$. See Figure \ref{splitting diagram}.
Then there exists a constant $C$ depending continuously on all the above parameters such that for all $f\in\Bes{s}{q}{\infty}(\R^d)$ and $\epsilon>0$, there exist functions $f^{1,\epsilon} \in \Bes{s_1}{p_1}{1}(\R^d)$ and $f^{2,\epsilon} \in \Bes{s_2}{p_2}{1}(\R^d)$ such that

\begin{equation}
f=f^{1,\epsilon} +f^{2,\epsilon} ,
\end{equation}
\begin{equation}
\norm{f^{1,\epsilon}}{\Bes{s_1}{p_1}{1}} \leq CN^{-\left((1-\frac{q}{p_2})\frac{s-\bar{s}}{\tilde{s}-\bar{s}}+\frac{q}{p_1}-1\right)}\norm{f}{\Bes{s}{q}{\infty}},
\end{equation}
\begin{equation}
\norm{f^{2,\epsilon}}{\Bes{s_2}{p_2}{1}} \leq CN^{(1-\frac{q}{p_2})\frac{\tilde{s}-s}{\tilde{s}-\bar{s}}}\norm{f}{\Bes{s}{q}{\infty}},
\end{equation}
\begin{equation}
\norm{f^{1,\epsilon}}{\Bes{s}{q}{\infty}},\norm{f^{2,\epsilon}}{\Bes{s}{q}{\infty}} \leq C\norm{f}{\Bes{s}{q}{\infty}}. \label{persistency1}
\end{equation}

\end{Proposition}
\ \\
See Figure \ref{splitting diagram}. For a proof of this result, see Proposition 5.6 of \cite{AlbrBarker}. 
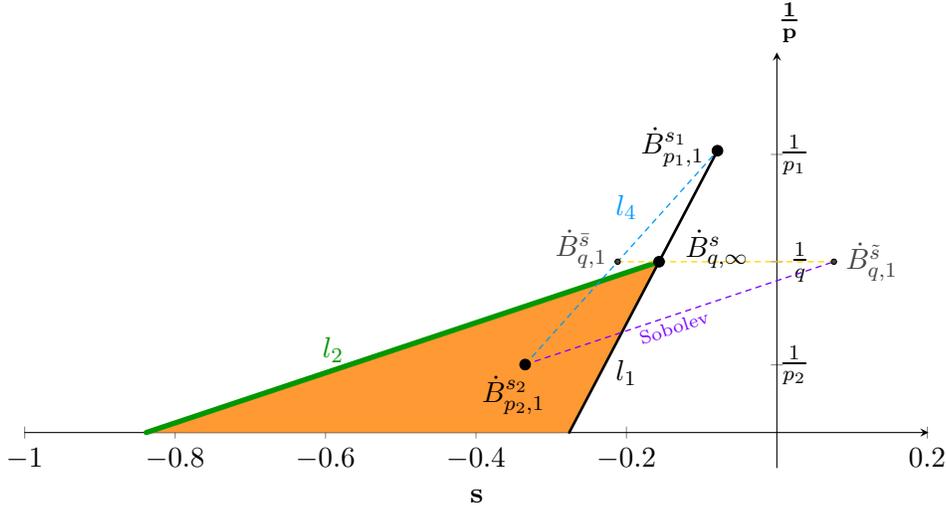
\begin{figure}[h] 
\centering
\begin{tikzpicture}[line cap=round,line join=round,>=triangle 45,x=10cm,y=10cm]
\begin{axis}[
x=10cm,y=10cm,
axis lines=middle,
xmin=-1,
xmax=0.20167597765363113,
ymin=-0.04676364650460901,
ymax=0.5058,
xtick={-1,-0.8,...,0.20000000000000004},
ytick={0.37,1/4.4,0.09},
yticklabels = {$\frac{1}{p_1}$,$\frac{1}{q}$,$\frac{1}{p_2}$},
xlabel near ticks,
ylabel near ticks,
xlabel={$\mathbf{s}$},
ylabel={$\mathbf{\frac{1}{p}}$},
ylabel style={rotate=-90, xshift=15pt, yshift=90pt},
yticklabel style={
    xshift=18pt
}
]
\clip(-1,-0.04676364650460901) rectangle (0.20167597765363113,0.5058);
\fill[line width=0pt,color=ffzztt,fill=ffzztt,fill opacity=0.34] (-0.15660337709620767,0.22727272727272727) -- (-0.8384215589143895,0) -- (-0.27606379124529246,0) -- cycle;
\draw [line width=2pt,color=qqwuqq] (-0.15660337709620767,0.22727272727272727)-- (-0.8384215589143895,0);
\draw [line width=0.5pt,dash pattern=on 2pt off 2pt,color=qqzzff] (-0.33463687150837995,0.09044514021178506)-- (-0.07896788898040852,0.3749737815751768);
\draw [line width=1pt] (-0.27606379124529246,0)-- (-0.07896788898040852,0.3749737815751768);
\draw [line width=0.5pt,dash pattern=on 2pt off 2pt,color=xfqqff] (-0.33463687150837995,0.09044514021178506)-- (0.07584588967444664,0.22727272727272724);
\draw [line width=0.5pt,dash pattern=on 2pt off 2pt,color=ffdxqq] (-0.21168766752666274,0.22727272727272724)-- (0.07584588967444664,0.22727272727272724);
\begin{scriptsize}
\draw [fill=black] (-0.07896788898040852,0.3749737815751768) circle (2pt);
\draw[color=black] (-0.13854748603351966,0.38014956216460644) node {$\Bes{s_1}{p_1}{1}$};
\draw [fill=black] (-0.33463687150837995,0.09044514021178506) circle (2pt);
\draw[color=black] (-0.35,0.05) node {$\Bes{s_2}{p_2}{1}$};
\draw [fill=black] (-0.15660337709620767,0.22727272727272727) circle (2pt);
\draw[color=black] (-0.07988826815642469,0.24443217530382524) node {$\Bes{s}{q}{\infty}$};
\draw[color=qqwuqq] (-0.59,0.11) node {$l_2$};
\draw[color=qqzzff] (-0.20,0.3) node {$l_4$};
\draw[color=ffzztt] (-0.5,0.0460759945072989) node {$D$};
\draw [fill=uuuuuu] (-0.21168766752666274,0.22727272727272724) circle (1pt);
\draw[color=uuuuuu] (-0.26,0.24443217530382524) node {$\Bes{\bar{s}}{q}{1}$};
\draw [fill=uuuuuu] (0.07584588967444664,0.22727272727272724) circle (1pt);
\draw[color=uuuuuu] (0.12625698324022333,0.22727272727272724) node {$\Bes{\tilde{s}}{q}{1}$};
\draw[color=black] (-0.2,0.08) node {$l_1$};
\draw[color=xfqqff] (-0.13687150837988837,0.14) node [rotate=18.43] {\tiny Sobolev};
\end{scriptsize}
\end{axis}
\end{tikzpicture}
\caption{Splitting diagram for Proposition \ref{genBesovsplit} \textit{This illustrates the splitting proof from \cite{AlbrBarker}. }}
\label{splitting diagram}
\end{figure}
\\ We apply this to split our initial data into an $L^2$ piece and a subcritical Besov piece, analogously to Calderón's splitting:
\pagebreak
\begin{Corollary} \label{usefulBesovsplit}
Let $2<q<\infty$, $-1+\frac{2}{q}<s<0$.\footnote{For the significance of this condition, see Figure \ref{supercritical splitting diagram}.} Then there exists $p_{q,s}>q$, $C_{q,s}>0$, $0<\delta<1-\frac{3}{p}$ and $\gamma_1,\gamma_2>0$ all depending on $q$ and $s$ such that for all (weakly) divergence-free $u_0 \in \Bes{s}{q}{\infty}(\R^3)$ and $\epsilon>0$, there exist (weakly) divergence-free functions $m_0^\epsilon \in \Bes{-1+\frac{3}{p}+\delta}{p}{p}(\R^3)\cap \Bes{s}{q}{\infty}(\R^3)$ and $w_0^\epsilon \in L^2(\R^3)\cap \Bes{s}{q}{\infty}(\R^3)$ such that:
\begin{align}
u_0 &= m_0^\epsilon + w_0^\epsilon,
\\ \norm{m_0^\epsilon}{\Bes{-1+\frac{3}{p}+\delta}{p}{p}} &\leq C_{q,s} \epsilon^{\gamma_1}\norm{u_0}{\Bes{s}{q}{\infty}},
\\ \norm{w_0^\epsilon}{L^2} &\leq C_{q,s} \epsilon^{-\gamma_2}\norm{u_0}{\Bes{s}{q}{\infty}},
\\ \norm{m_0^\epsilon}{\Bes{s}{q}{\infty}},\norm{w_0^\epsilon}{\Bes{s}{q}{\infty}} &\leq C_{q,s}\norm{u_0}{\Bes{s}{q}{\infty}}.
\end{align}
\end{Corollary}
\begin{proof}
We choose the parameters of the previous Proposition \ref{genBesovsplit} carefully. Set $p_1=2$ and $s_1 = 0$. We will choose $\left(s_2,\frac{1}{p}\right)$ inside the region $D$. Let us first define some important lines:
\begin{equation}
l_{crit} = \left\lbrace\left(-1+\frac{3}{\alpha},\frac{1}{\alpha}\right) : \alpha\geq 3 \right\rbrace, \qquad  l_{bndry} = \left\lbrace\left(-1+\frac{2}{\alpha},\frac{1}{\alpha}\right) : \alpha\geq 2 \right\rbrace,
\end{equation}
\begin{equation} \nonumber
l_{1} = \left\lbrace\left(x,\frac{1}{2}+\frac{\left(-1+\frac{2}{q}\right)x}{2s}\right) : x\in \R \right\rbrace, \qquad l_2 = \left\lbrace\left(x,\frac{1}{q}-\frac{s}{3}+\frac{x}{3}\right) : x\in \R \right\rbrace .
\end{equation}
\ \\
We now consider two cases:
\\
\textbf{Case 1: subcritical or critical:} $q>3,-1+\frac{3}{q}\leq s<0$. See Figure \ref{subcritical splitting diagram}:
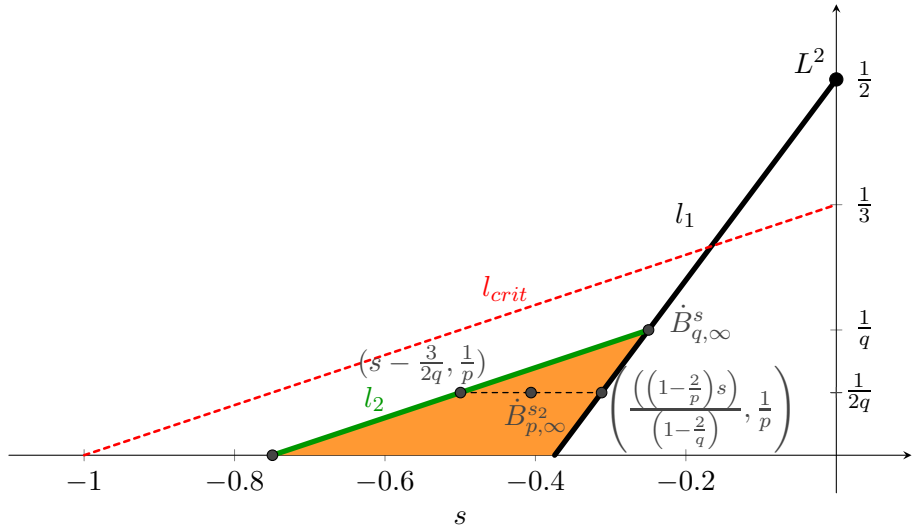
\begin{figure}[ht] 
\centering
\begin{tikzpicture}[line cap=round,line join=round,>=triangle 45,x=10cm,y=10cm]
\begin{axis}[
x=10cm,y=10cm,
axis lines=middle,
xmin=-1.1,
xmax=0.1,
ymin=-0.05,
ymax=0.6,
xtick={-1,-0.8,...,0},
ytick={0,1/3,1/2,1/6,1/12},
ylabel near ticks,
yticklabels = {$0$,$\frac{1}{3}$,$\frac{1}{2}$,$\frac{1}{q}$,$\frac{1}{2q}$},
xlabel near ticks,
xlabel={$s$},
yticklabel style={
    xshift=20pt
}]
\clip(-1.1,-0.2) rectangle (0.2,0.6);
\fill[line width=0pt,color=ffzztt,fill=ffzztt,fill opacity=0.38] (-0.74962962962963,0) -- (-0.37444444444444497,0) -- (-0.24962962962963,0.16666666666666666) -- cycle;
\draw [line width=2pt] (-0.37444444444444497,0)-- (0,0.5);
\draw [line width=2pt,color=qqwuqq] (-0.74962962962963,0)-- (-0.24962962962963,0.16666666666666666);
\draw [line width=1pt,dash pattern=on 2pt off 2pt,color=ffqqqq] (-1,0)-- (0,0.3333333333333333);
\draw [line width=0.5pt,dash pattern=on 2pt off 2pt] (-0.49962962962963,0.08333333333333333)-- (-0.3120370370370375,0.08333333333333333);
\begin{scriptsize}
\draw [fill=uuuuuu] (-0.24962962962963,0.16666666666666666) circle (2pt);
\draw[color=uuuuuu] (-0.18,0.17443181818181816) node {$\Bes{s}{q}{\infty}$};
\draw [fill=uuuuuu] (-0.49962962962963,0.08333333333333333) circle (2pt);
\draw[color=uuuuuu] (-0.55,0.11988636363636362) node {$(s - \frac{3}{2q}, \frac{1}{p})$};
\draw [fill=uuuuuu] (-0.3120370370370375,0.08333333333333333) circle (2pt);
\draw[color=uuuuuu] (-0.18,0.06079545454545454) node {$\left(\frac{\left(\left(1 - \frac{2}{p}\right) s\right)}{\left(1 - \frac{2}{q}\right)}, \frac{1}{p}\right)$};
\draw [fill=uuuuuu] (-0.40583333333333377,0.08333333333333333) circle (2pt);
\draw[color=uuuuuu] (-0.4,0.05) node {$\Bes{s_2}{p}{\infty}$};
\draw [fill=uuuuuu] (-0.74962962962963,0) circle (2pt);
\draw[color=ffzztt] (-0.5,0.03) node {$D$};
\draw [fill=black] (0,0.5) circle (2.5pt);
\draw[color=black] (-0.035,0.525) node {$L^2$};
\draw[color=black] (-0.2,0.32) node {$l_1$};
\draw[color=qqwuqq] (-0.6134078212290504,0.075) node {$l_2$};
\draw[color=ffqqqq] (-0.44,0.2221590909090909) node {$l_{crit}$};
\draw [color=uuuuuu] (0,0.08333333333333333)-- ++(-2pt,0 pt) -- ++(4pt,0 pt) ++(-2pt,-2pt) -- ++(0 pt,4pt);

\end{scriptsize}
\end{axis}
\end{tikzpicture}
\caption{Subcritical/critical splitting diagram, corresponding to Proposition \ref{genBesovsplit}/Figure \ref{splitting diagram}. \textit{Observe that the highlighted orange region $D$ is contained fully in the subcritical region under the red line $l_{crit}$. This means that any point in $D$ would be viable for our purposes. We choose $p=2q$ and bisect between $l_1$ and $l_2$ for simplicity.}}
\label{subcritical splitting diagram}
\end{figure}

We may take $p=2q$. We see that $\left(\frac{s(q-1)}{\left(q-2\right)} ,\frac{1}{2q}\right) \in l_1$ and $\left(s-\frac{3}{2q} ,\frac{1}{2q}\right) \in l_2$. 
Thus, a suitable choice for $s_2$ is $s_2:=\frac{\frac{s(q-1)}{\left(q-2\right)}+s-\frac{3}{2q}}{2}=-\frac{3}{4q}+\frac{s}{2}\left(\frac{2q-3}{q-2} \right)$. 
To obtain the splitting we desire, $\delta$ is defined by the relation $-1+\frac{3}{2q}+\delta=\frac{\frac{s(q-1)}{\left(q-2\right)}+s-\frac{3}{2q}}{2}=-\frac{3}{4q}+\frac{s}{2}\left(\frac{2q-3}{q-2} \right)$. Hence, we have
\begin{equation}
\delta = 1-\frac{9}{4q}+ \frac{s(2q-3)}{2(q-2)}. \label{deltavaluesub}
\end{equation}
By construction, it is clear that $\delta \in (0,1-\frac{3}{2q})$.
Finally, we set \begin{equation}
\bar{s}= \frac{\frac{1}{2}-\frac{1}{q}}{\frac{1}{2}-\frac{1}{2q}}\left(\frac{s(2q-3)}{2(q-2)}-\frac{3}{4q}\right) = \frac{1}{2(q-1)}\left(s(2q-3) - \frac{3(q-2)}{2q} \right)
\end{equation}
and
\begin{equation}
 \tilde{s} = s_2+\frac{3}{q}-\frac{3}{p}= \frac{3}{4q}+\frac{s(2q-3)}{2(q-2)}.
\end{equation} 
\ \\
\textbf{Case 2: supercritical:} $-1+\frac{2}{q}<s<\min\left(0,-1+\frac{3}{q}\right)$. 
See Figure \ref{supercritical splitting diagram}: 
\begin{figure}[ht] 
\centering
\begin{tikzpicture}[line cap=round,line join=round,>=triangle 45,x=9cm,y=9cm]

\begin{axis}[
x=9cm,y=9cm,
axis lines=middle,
xmin=-1.3534217877094972,
xmax=0.09657821229050288,
ymin=-0.07,
ymax=0.58,
xtick={-1.2000000000000002,-1.0000000000000002,...,0},
ytick={0,1/3,1/2,1/2.51,0.07342401638145372,0.21277918761324555},
ylabel near ticks,
yticklabels = {$0$,$\frac{1}{3}$,$\frac{1}{2}$,$\frac{1}{q}$,$\frac{1}{3P}$,$\frac{1}{P}$},
xlabel near ticks,
xlabel={$s$},
yticklabel style={
    xshift=22pt
}]

\fill[line width=0pt,color=ffzztt,fill=ffzztt,fill opacity=0.38] (-1.318406374501992,0) -- (-0.6062745098039221,0) -- (-0.123187250996016,0.39840637450199207) -- cycle;
\draw [line width=2pt] (-0.6062745098039221,0)-- (0,0.5);
\draw [line width=2pt,color=qqwuqq] (-1.318406374501992,0)-- (-0.123187250996016,0.39840637450199207);
\draw [line width=1pt,dash pattern=on 2pt off 2pt,color=ffqqqq] (-1,0)-- (0,0.3333333333333333);
\draw [line width=1pt,dash pattern=on 2pt off 2pt,color=qqqqff] (-1,0)-- (0,0.5);
\draw [line width=0.5pt,dash pattern=on 2pt off 2pt] (-0.7797279508556388,0.07342401638145372)-- (-0.5172442907249201,0.07342401638145372);
\begin{scriptsize}
\draw [fill=uuuuuu] (-0.123187250996016,0.39840637450199207) circle (2pt);
\draw[color=uuuuuu] (-0.05328212290502783,0.3847859621884539) node {$\Bes{s}{q}{\infty}$};
\draw [fill=uuuuuu] (-0.3391838525669165,0.2202720491443612) circle (2pt);
\draw[color=uuuuuu] (-0.17884078212290494,0.21277918761324555) node {$\left(-1 + \frac{3}{P}, \frac{1}{P}\right)$};
\draw [fill=uuuuuu] (-0.7797279508556388,0.07342401638145372) circle (2pt);
\draw [fill=uuuuuu] (-0.5172442907249201,0.07342401638145372) circle (2pt);
\draw[color=uuuuuu] (-0.35,0.056243921862262236) node {$\left(\frac{\left(\left(1 - \frac{2}{p}\right) s\right)}{\left(1 - \frac{2}{q}\right)}, \frac{1}{p}\right)$};
\draw [fill=uuuuuu] (-0.6484861207902795,0.07342401638145372) circle (2pt);
\draw[color=uuuuuu] (-0.6405726256983238,0.04) node {$\Bes{s_2}{p}{\infty}$};
\draw[color=ffzztt] (-1.1,0.035) node {$D$};
\draw [fill=black] (0,0.5) circle (1pt);
\draw[color=black] (-0.04,0.53) node {$L^2$};
\draw[color=black] (-0.18491620111731835,0.3110687730847932) node {$l_1$};
\draw[color=qqwuqq] (-0.7985335195530725,0.2045883888239499) node {$l_2$};
\draw[color=ffqqqq] (-0.4370460893854748,0.225) node {$l_{crit}$};
\draw[color=qqqqff] (-0.4370460893854748,0.35) node {$l_{bndry}$};
\draw [color=uuuuuu] (0,0.07342401638145372)-- ++(-2pt,0 pt) -- ++(4pt,0 pt) ++(-2pt,-2pt) -- ++(0 pt,4pt);
\end{scriptsize}
\end{axis}
\end{tikzpicture}
\caption{Supercritical splitting diagram, corresponding to Proposition \ref{genBesovsplit}/Figure \ref{splitting diagram}. \textit{Notice here that the highlighted orange region $D$ is not fully contained in the subcritical region under the red line $l_{crit}$. This means that we must take care to choose a point in $D$ that corresponds to a subcritical Besov space. 
\\ We can see the significance of the condition $-1+\frac{2}{q}<s<0$, which corresponds to the region below the blue line $l_{bndry}$; if we have our initial data in a Besov space corresponding to a point above $l_{bndry}$, then it is clear that $l_1$ (the line defined through $\frac{1}{2}$ (corresponding to $L^2$) and $(s,q)$ (corresponding to $\Bes{s}{q}{\infty}$)) does not interest with $l_{crit}$, and hence the region $D$ would not contain any subcritical points.} } 
\label{supercritical splitting diagram}
\end{figure}
\\We must take more care choosing a point inside $D$ that corresponds to a subcritical Besov space. We observe that $l_1$ and $l_\text{crit}$ intersect when $P:=\frac{2s+3-\frac{6}{q}}{s+1-\frac{2}{q}}$. Accordingly, we choose \begin{equation}
p=3P=\frac{3\left(2s+3-\frac{6}{q}\right)}{s+1-\frac{2}{q}}. \label{pvaluesuper}
\end{equation}
We see that $\left( \frac{\left( 1-\frac{2}{p} \right)s}{1-\frac{2}{q}} ,\frac{1}{p}\right) \in l_1$ and $\left(-1+\frac{3}{p} ,\frac{1}{p}\right) \in l_\text{crit}$.
Thus, a suitable choice for $s_2$ is 
\begin{equation}
s_2:= \left( 1-\frac{2}{p} \right)\left(1-\frac{2}{q}\right)^{-1}\frac{s}{2} - \frac{1}{2}\left(1-\frac{3}{p}\right) 
\end{equation}
To obtain the splitting we desire, $\delta$ is defined by the relation $s_2=-1+\frac{3}{p}+\delta$.
Hence, we have
\begin{equation}
\delta := \frac{1}{2}\left(1-\frac{3}{p}\right) + \left( 1-\frac{2}{p} \right)\left(1-\frac{2}{q}\right)^{-1}\frac{s}{2} . \label{deltavaluesuper}
\end{equation}
By construction, it is clear that $\delta \in \left(0,1-\frac{3}{p}\right)$.
Finally, we set \begin{equation}
\bar{s}= \frac{\frac{1}{q}-\frac{1}{2}}{\frac{1}{p}-\frac{1}{2}}\left(-1+\frac{3}{p}+\delta\right)
\quad \text{and} \quad
 \tilde{s} = s_2+\frac{3}{q}-\frac{3}{p}= -1+\frac{3}{q}+\delta. 
\end{equation}
\\  
In both cases, we may then apply Proposition \ref{genBesovsplit}: there exists $C_{q,s},\gamma_1,\gamma_2>0$ depending on $q$ and $s$ such that for any $u_0 \in \Bes{s}{q}{\infty}$ and $\epsilon>0$, there exist functions $u_0^{1,\epsilon}$ and $u_0^{2,\epsilon}$, such that
\begin{equation}
u_0=u_0^{1,\epsilon} +u_0^{2,\epsilon} ,
\end{equation}
\begin{equation}
\norm{u_0^{1,\epsilon}}{L^2} = \norm{u_0^{1,\epsilon}}{\Bes{0}{2}{2}} \leq  C\norm{u_0^{1,\epsilon}}{\Bes{0}{2}{1}} \leq C_{q,s}\epsilon^{-\gamma_2}\norm{u_0}{\Bes{s}{q}{\infty}},
\end{equation}
\begin{equation}
\norm{u_0^{2,\epsilon}}{\Bes{-1+\frac{3}{p}+\delta}{p}{p}} \leq C_{q,s}\norm{u_0^{2,\epsilon}}{\Bes{-1+\frac{3}{p}+\delta}{p}{1}} \leq C_{q,s}\epsilon^{\gamma_1}\norm{u_0}{\Bes{s}{q}{\infty}},
\end{equation}
\begin{equation}
\norm{u_0^{1,\epsilon}}{\Bes{s}{q}{\infty}},\norm{u_0^{2,\epsilon}}{\Bes{s}{q}{\infty}} \leq C_{q,s}\norm{u_0}{\Bes{s}{q}{\infty}}.
\end{equation}
\\ We may then apply the Leray projector and fact that the Leray projector is a bounded linear operator for these homogeneous Besov spaces (see Proposition 2.30 of \cite{Bahouri}) to obtain $m_0^\epsilon$ from $u_0^{2,\epsilon}$ and $w_0^\epsilon$ from $u_0^{1,\epsilon}$ for the desired splitting, where we set $\gamma_1,\gamma_2>0$:
\begin{align}
\gamma_1 = \frac{\tilde{s}-s}{(\tilde{s}-\bar{s})}\left(\frac{p-q}{p}\right),
\qquad \gamma_2 = \frac{s-\bar{s}}{(\tilde{s}-\bar{s})}\left(\frac{p-q}{p}\right)+\frac{q}{2}-1.
\end{align}
\end{proof}
\begin{Remark}
    In the subcritical/critical case of Corollary \ref{usefulBesovsplit}, the following expression is useful for some applications (in particular, for Theorem \ref{sequential scaleinvbnd}):  
\begin{align}\label{Specificratiosub}
\frac{\gamma_2}{\gamma_1} = 2(q-2)
\end{align}
To obtain this formula, we see that
\begin{align}
1+\frac{\gamma_2}{\gamma_1} &= \frac{\tilde{s}-s+s-\bar{s}+\frac{p}{p-q}\left( \frac{q}{2}-1 \right)(\tilde{s}-\bar{s})}{\tilde{s}-s} = \frac{\tilde{s}-\bar{s}}{\tilde{s}-s}\left( \frac{q(p-2)}{2(p-q)} \right)
\end{align}
and
\begin{align}
\frac{\tilde{s}-\bar{s}}{\tilde{s}-s} = \frac{\frac{3}{4q}+\frac{s(2q-3)}{2(q-2)}-\frac{1}{2(q-1)}\left(s(2q-3) - \frac{3(q-2)}{2q} \right)}{\frac{3}{4q}+\frac{s(2q-3)}{2(q-2)}-s} = \frac{2q-3}{q-1}.
\end{align}
Then using that $p=2q$ in the subcritical/critical case of Corollary \ref{usefulBesovsplit}, we obtain the formula for (\ref{Specificratiosub}) for $\frac{\gamma_2}{\gamma_1}$.
\end{Remark}

\section{Global-in-time existence of weak solutions corresponding to supercritical Besov space initial data}
In this section, we prove Theorem \ref{Global solution supercritical Besov data}. In fact, we will prove:
\setcounter{MainTheorem}{1}
\begin{Theorem}
     Let $q>2$ and $-1+\frac{2}{q}<s<\min \left(-1+\frac{3}{q},0 \right).$ Suppose that $u_0 \in \Bes{s}{q}{\infty}(\R^3)$ is divergence-free. Then there exist constants $\max(q,3)< p<\infty$ and $0<\delta<1-\frac{3}{p}$ (both depending on $q$ and $s$) and a suitable\footnote{As in Definition \ref{suitablesolngen}.} weak solution $(u,\pi)$ to the Navier-Stokes equations on $\R^3\times(0,\infty)$ which may be decomposed as $$u=m+w \in K^\infty(\infty)\cap K^p(\infty)\cap K^{p,\delta}(\infty) + L^\infty_tL^2_x\cap L^2_t\dot{H}^1_x(\R^3\times(0,\infty)),$$ corresponding to an initial data splitting $$u_0  = m_0 + w_0 \in \Bes{-1+\frac{3}{p}+\delta}{p}{p}\cap \Bes{s}{q}{\infty} + L^2\cap \Bes{s}{q}{\infty}(\R^3) $$   
     such that $m$ is a mild solution of (\ref{Navier-Stokes equations}) with initial data $m_0$ and $w$ satisfies: 
    \begin{enumerate}
        \item[(i)] $\lim\limits_{t\to0^+}\norm{w(\cdot,t)-w_0}{L^2}=0$.
\item[(ii)] for all $\phi \in C^\infty_c(\R^3),$ the map $t \mapsto \int_{\R^3}w(x,t)\cdot\phi(x)\;dx$ is continuous on $[0,\infty).$ 
\item[(iii)]w satisfies the (perturbed) energy inequality: for almost every $t_1 \in \left[0,\infty\right)$ (and including $t_1=0$) and for all $t_2 \in \left[ t_1, \infty\right)$, 
    \begin{equation}
        \norm{w(\cdot,t_2)}{L^2_x}^2 +2\int_{t_1}^{t_2} \norm{\nabla w(\cdot,s)}{L^2_x}^2ds \leq \norm{w(\cdot,t_1)}{L^2_x}^2 + 2\int_{t_1}^{t_2} \int_{\R^3} m\cdot(w\cdot\nabla)w\;dxds. 
    \end{equation}
\item[(iv)] $w$ is suitable in the sense of Definition \ref{suitablesolngen} for the $m$-perturbed Navier-Stokes equations on $\R^3\times (0,\infty).$
\end{enumerate}      
\end{Theorem}

\begin{proof}
\textbf{Step 1: Split the initial data.}  We begin by splitting the initial data as in Proposition \ref{usefulBesovsplit}: for $\epsilon>0$ (which we fix later) there exists $p>\max(q,3)$, $0<\delta<1-\frac{3}{p}$ and $\gamma_1, \gamma_2, C_{q,s}>0$ all depending on $q$ and $s$ such that $u_0 = m_0 + w_0$ with $m_0 \in \Bes{-1+\frac{3}{p}+\delta}{p}{p}\cap \Bes{s}{q}{\infty}$ and $w_0 \in L^2\cap \Bes{s}{q}{\infty}$ (weakly) divergence-free such that:
\begin{align}
\norm{m_0}{\Bes{-1+\frac{3}{p}+\delta}{p}{\infty}} &\leq C_{q,s} \epsilon^{\gamma_1}\norm{u_0}{\Bes{s}{q}{\infty}}, \label{subcritpart}
\\ \norm{w_0}{L^2} &\leq C_{q,s} \epsilon^{-\gamma_2}\norm{u_0}{\Bes{s}{q}{\infty}}, \label{L2part}
\\ \norm{m_0}{\Bes{s}{q}{\infty}},\norm{w_0}{\Bes{s}{q}{\infty}} &\leq C_{q,s}\norm{u_0}{\Bes{s}{q}{\infty}}. \label{persistencyline}
\end{align}
Define $\delta^*:=-1+\frac{3}{q}-s>0$. By using the Besov space embedding (Proposition \ref{Besov embeddings}) and (\ref{persistencyline}):
\begin{align}
    \norm{m_0}{\Bes{-1+\frac{3}{p}-\delta^*}{p}{\infty}} = \norm{m_0}{\Bes{s-3\left(\frac{1}{q}-\frac{1}{q}\right)}{p}{\infty}} \lesssim_{q,s}\norm{m_0}{\Bes{s}{q}{\infty}} \lesssim_{q,s}\norm{u_0}{\Bes{s}{q}{\infty}}. \label{supcritpart}
\end{align}
Using Besov space interpolation (Proposition \ref{Besovinterpolation}) with $-1+\frac{3}{p}\in \left(-1+\frac{3}{p}-\delta^*,-1+\frac{3}{p}+\delta\right)$ and lines (\ref{subcritpart}) and (\ref{supcritpart}) we may then see that the critical norm satisfies:
\begin{align}
    \norm{m_0}{\Bes{-1+\frac{3}{p}}{p}{\infty}} &\leq \norm{m_0}{\Bes{-1+\frac{3}{p}-\delta^*}{p}{\infty}}^\theta\norm{m_0}{\Bes{-1+\frac{3}{p}+\delta}{p}{\infty}}^{1-\theta} 
    \\ &\lesssim_{q,s}\norm{u_0}{\Bes{s}{q}{\infty}}^\theta \epsilon^{\gamma_1(1-\theta)}\norm{u_0}{\Bes{s}{q}{\infty}}^{1-\theta}
    \\ &\lesssim_{q,s}\epsilon^{\gamma_1(1-\theta)}\norm{u_0}{\Bes{s}{q}{\infty}} \label{critpart}
\end{align}
where $\theta = \frac{\delta}{\delta+\delta^*} \in (0,1)$ (note that $\theta$ depends on $q$ and $s$).
\\
\\
\textbf{Step 2: Construct a mild solution corresponding to the subcritical part.} \\We set $\epsilon := \left( \frac{C_{q,s}^{(1)}}{2C_{q,s}\norm{u_0}{\Bes{s}{q}{\infty}}} \right)^{\frac{1}{\gamma_1(1-\theta)}}$ where $C_{q,s}^{(1)}=C_{p,0}^{(1)}$ comes from Theorem \ref{Besovmildsoln} and $C_{q,s}$ is the implicit constant that arises in line (\ref{critpart}). Viewing $m_0 \in \Bes{-1+\frac{3}{p}}{p}{\infty}$, this choice of $\epsilon$ ensures that we may apply Theorem \ref{Besovmildsoln} with ``$\delta=0$ and $T=\infty$": there exists a mild solution $m \in K^p(\infty)\cap K^\infty(\infty)$ to the Navier-Stokes equations with initial data $m_0$ such that 
\begin{equation}
    \norm{m}{K^p(\infty)\cap K^\infty(\infty)} := \norm{m}{K^p(\infty)}+\norm{m}{K^\infty(\infty)} \leq C^{(2)}_{p,0}=:C^{(2)}_{q,s}. 
\end{equation}
We exploit the fact that we also have $m_0 \in \Bes{-1+\frac{3}{p}+\delta}{p}{\infty}$ to obtain additional information about $m$ by applying Lemma \ref{contractionLemma} (ii). Using the same notation as in Theorem \ref{Besovmildsoln}, suppose that $a\in K^{p,\delta}(\infty)$ and $b\in K^p(\infty)\cap K^\infty(\infty)$. For any $0<t<\infty$, using Young's inequality, the Oseen kernel estimate (\ref{Oseen Lp bnd}) and Hölder's inequality, we see:
\begin{align}
    \norm{B(a,b)(\cdot,t)}{L^p} &\leq \int_0^t \norm{K}{L^1}\norm{(a \otimes b)(\cdot,s)}{L^p}ds 
    \\ &\lesssim \int_0^t\frac{1}{(t-s)^\frac{1}{2}} \norm{a(\cdot,s)}{L^p}\norm{b(\cdot,s)}{L^\infty}ds
    \\ &\lesssim \norm{a}{K^{p,\delta}(\infty)}\norm{b}{K^\infty(\infty)}\int_0^t\frac{1}{(t-s)^\frac{1}{2}\cdot s^{\frac{1}{2}\left(1-\frac{3}{p}-\delta\right)}\cdot s^\frac{1}{2}}ds
    \\ &\lesssim_{p,\delta}\norm{a}{K^{p,\delta}(\infty)}\norm{b}{K^\infty(\infty)}t^{-\frac{1}{2}\left(1-\frac{3}{p}-\delta\right)}.
\end{align}
Therefore, 
\begin{equation}
    \norm{B(a,b)}{K^{p,\delta}(\infty)} \leq C_{p,\delta}\norm{a}{K^{p,\delta}(\infty)}\norm{b}{K^p(\infty)\cap K^\infty(\infty)}.
\end{equation}
This allows\footnote{Taking care to compare $C_{p,\delta} \leq \lambda_{p,0}$ from line (\ref{bilinearKpT}); indeed, we may make $\lambda_{p,0}$ greater if we need.} us to apply Lemma \ref{contractionLemma} (ii) to find that \begin{equation}
    m \in K^p(\infty)\cap K^\infty(\infty)\cap K^{p,\delta}(\infty)
\end{equation} with 
\begin{equation}
    \norm{m}{K^{p,\delta}(\infty)} \lesssim \norm{e^{t\Delta}m_0}{K^{p,\delta}(\infty)} . 
\end{equation}
Using our splitting, we have that 
\begin{equation}
    \norm{m}{K^{p,\delta}(\infty)} \lesssim \norm{e^{t\Delta}m_0}{K^{p,\delta}(\infty)} \lesssim_{q,s}\norm{m_0}{\Bes{-1+\frac{3}{p}+\delta}{p}{\infty}} \lesssim_{q,s} \epsilon^{\gamma_1}\norm{u_0}{\Bes{s}{q}{\infty}}= C_{q,s}\norm{u_0}{\Bes{s}{q}{\infty}}^{-\frac{\theta}{1-\theta}}. \label{subcritkatoestim}
\end{equation}

We also note that as $m\in K^\infty(\infty),$ we know that $m$ is smooth - see for example, Section 6.4 of \cite{sereginlecturenotes}.
\\
\\
\textbf{Step 3: Consider the perturbed Navier-Stokes Equations corresponding to $w_0$.}
\\ 
We now consider the $m$-perturbed Navier Stokes equations
\begin{equation}
 \label{mPerturbed Navier-Stokes equations}
    \partial_t w -\Delta w +\text{div}(w\otimes w + m \otimes w + w \otimes m) +\nabla \tilde{\pi} =0,
    \qquad \text{div}( w) = 0
\end{equation}
with initial data $w_0.$
\\ For any $T>0$, we note that $m\in L^r_tL^p_x((0,T)\times \R^3)$ where $r>2$ is defined by the relation $\frac{2}{r}+\frac{3}{p}=1$. Indeed,
\begin{align}
    \norm{m}{L^r_tL^p_x((0,T)\times \R^3)}^r &\leq \int_0^T t^{-\frac{r}{2}\left(1-\frac{3}{p}-\delta\right)} \norm{m}{K^{p,\delta}(\infty)}^r dt =T^\frac{\delta r}{2} \norm{m}{K^{p,\delta}(\infty)}^r. \label{katoestim Lr}
\end{align}

This allows us to use the results presented in Albritton's paper \cite{AlbBesovBlow} off-the-shelf. Precisely, for any fixed $T>0$, Proposition 2.3 of \cite{AlbBesovBlow} states that there exists a weak solution $w$ to the $m$-perturbed Navier-Stokes equations such that:
\begin{enumerate}
    \item[(a)] $w\in L^\infty_tL^2_x\cap L^2_t\dot{H}^1_x((0,T)\times \R^3)$ and $\tilde{\pi} \in  L^2_tL^\frac{3}{2}_x + L^2_tL^2_x((0,T)\times \R^3)$.
    \item[(b)] $\lim_{t\downarrow0^+}\norm{w(\cdot,t)-w_0}{L^2}=0$.
    \item[(c)] $(w,\tilde{\pi})$ is suitable (in the sense of Definition \ref{suitablesolngen}) on $(0,T)\times \R^3$.
    \item[(d)] $w$ satisfies the global energy inequality: for almost every $t_1 \in \left[0,T\right]$ (and including $t_1=0$) and for all $t_2 \in \left[ t_1, T\right]$, the following inequality holds true:
    \begin{equation}
        \norm{w(t_2)}{L^2_x}^2 +2\int_{t_1}^{t_2} \norm{\nabla w(s)}{L^2_x}^2ds \leq \norm{w(t_1)}{L^2_x}^2 + 2\int_{t_1}^{t_2} \int_{\R^3} m\cdot(w\cdot\nabla)w\;dxds.
    \end{equation}
    \item[(e)] for all $\phi \in C^\infty_c(\R^3),$ the map $t \mapsto \int_{\R^3}w(x,t)\cdot\phi(x)\;dx$ is continuous on $[0,T].$
\end{enumerate}
In particular, for all $t\in [0,T],$ we may estimate, by using Hölder's inequality and Young's inequality,
\begin{align}
\norm{w(t)}{L^2_x}^2 + 2\int_0^t \norm{\nabla w(s)}{L^2_x}^2 ds \leq \norm{w_{0}}{L^2_x}^2 &+ 2\int_0^t C_p\norm{m(s)}{L^{p}_x}^r\norm{w(s)}{L^2_x}^2 ds+ \int_0^t\norm{\nabla w(s)}{L^2_x}^2 ds. \label{Calderoncontrol2a}
\end{align}
where $\frac{2}{r}+\frac{3}{p}=1.$
Applying Gronwall's inequality yields:
\begin{equation} \label{gronreskato2a}
\norm{w(t)}{L^2_x}^2 + \int_0^t \norm{\nabla w(s)}{L^2_x}^2 ds \leq \norm{w_{0}}{L^2_x}^2\exp\left(C_p\int_0^t \norm{m(s)}{L^{p}_x}^r ds \right). 
\end{equation}
Using (\ref{subcritkatoestim}), (\ref{katoestim Lr}) and the above estimates, we then obtain
\begin{align}
\norm{w}{L^\infty_t L^2_x(\R^3\times(0,T)}^2 + \int_0^T \norm{\nabla w(s)}{L^2_x}^2 ds &\leq \norm{w_{0}}{L^2_x}^2\exp\left(C_{q,s} T^{\frac{\delta r}{2}} \norm{u_0}{\Bes{s}{q}{\infty}}^{-\frac{\theta r}{1-\theta}}\right). \label{Calderonenergyestim}
\end{align}

Using this energy bound and a Cantor diagonalisation argument, we may construct a solution $w$ to the $m$-perturbed Navier-Stokes equations on $(0,\infty)\times \R^3$, satisfying properties (a)-(e). 
\\ Finally, we define $u:= m+w$ and $\pi:= \pi_m +\tilde{\pi}$, where $\pi_m$ is the pressure corresponding to $m$. Using the properties of $m$ (in particular, $m$ is smooth) and the fact that $w$ is a suitable weak solution to the $m$-perturbed equation, we finally see that $u$ is a suitable weak solution to the Navier-Stokes equations on $(0,\infty)\times \R^3$.
\end{proof}
\newpage
\section{Quantifying the number of singular points}
\subsection{Splitting a Leray-Hopf weak solution}
\begin{Lemma} \label{Besovenergy2}
Let $T \in (0,\infty)$, $2<q<\infty$, $-1+\frac{2}{q}<s<0$. There exist constants $p>3$, $\delta\in \left(0,1-\frac{3}{p}\right)$, $\gamma_1,\gamma_2>0$ all depending on $q$ and $s$ such that the following holds:
Let $u$ be a weak Leray-Hopf solution to the Navier-Stokes equations corresponding to divergence-free initial data $u_0\in\Bes{s}{q}{\infty}(\R^3)\cap L^2(\R^3)$ on the time interval $[0,\infty)$.
Then we may conclude that $u=m+w$ on $\R^3 \times (0,T)$ where $m$ and $w$ satisfy the following estimates: 
\begin{equation} \label{mkato2}
\norm{m}{K^{p,\delta}(T)}+\norm{m}{K^{\infty,\delta}(T)} \lesssim _{q,s}T^{-\frac{\delta}{2}},
\end{equation}
\begin{equation} \label{wenergy2}
\sup_{0<t<T} \norm{w(t)}{L^2_x}^2 + \int_0^T \norm{\nabla w(s)}{L^2_x}^2 ds \lesssim_{q,s}T^\frac{\delta\gamma_2}{\gamma_1}\norm{u_0}{\Bes{s}{q}{\infty}}^{2(\frac{\gamma_2}{\gamma_1}+1)}.
\end{equation} 
Here, $m$ is the unique weak Leray solution corresponding to partial initial data $m_0\in \Bes{-1+\frac{3}{p}+\delta}{p}{\infty}(\R^3)$ (as defined below in (\ref{westimkato2})), and exists smoothly on $(0,T)$. 

\end{Lemma}
\begin{proof}
\textbf{Step 1: split the initial data and construct a mild solution:}
Given $\epsilon>0$ to be determined, we use Corollary \ref{usefulBesovsplit} to decompose the initial data into $u_0=m_0^\epsilon + w_0^\epsilon$ so that $m_0^\epsilon$ and $w_0^\epsilon$ satisfy, for some $\delta, C_{q,s},\gamma_1,\gamma_2>0,p>q$ all depending on $q$ and $s$,
\begin{gather}
\norm{m_0^\epsilon}{\Bes{-1+\frac{3}{p}+\delta}{p}{\infty}} \leq C_{q,s}\norm{m_0^\epsilon}{\Bes{-1+\frac{3}{p}+\delta}{p}{p}} \leq C_{q,s} \epsilon^{\gamma_1}\norm{u_0}{\Bes{s}{q}{\infty}}, \label{mestimkato2}
\\ \norm{w_0^\epsilon}{L^2} \leq C_{q,s} \epsilon^{-\gamma_2}\norm{u_0}{\Bes{s}{q}{\infty}} \label{westimkato2}.
\end{gather}
We note from one of (\ref{deltavaluesub}) or (\ref{deltavaluesuper}) that $\delta$ takes a specific value depending on $q$ and $s$, and lies in the non-empty range $\left(0,1-\frac{3}{p}\right)$. 
We take \begin{equation} \label{epsilonchoicekato2}
\epsilon := \left(\dfrac{C_{q,s}^{(1)}}{C_{q,s}T^\frac{\delta}{2}\norm{u_0}{\Bes{s}{q}{\infty}}}\right)^\frac{1}{\gamma_1}
\end{equation}
so\footnote{Where $C_{q,s}^{(1)}$ is the constant from (\ref{Besovinfty}) in Theorem \ref{Besovmildsoln}, recalling $\delta$ and $p$ depend now on $q$ and $s$.} that by (\ref{mestimkato2}), \begin{equation}
T^\frac{\delta}{2}\norm{m_0^\epsilon}{\Bes{-1+\frac{3}{p}+\delta}{p}{\infty}} \leq C_{q,s}^{(1)}.
\end{equation}
This allows us to apply Theorem \ref{Besovmildsoln} to deduce that there exists $m^\epsilon \in K^{p,\delta}(T)\cap K^{\infty,\delta}(T)$, a mild solution on $(0,T)$ corresponding to $m_0^\epsilon$. With this specific choice of $\epsilon$, we now denote $m=m^\epsilon$,  $m_0=m^\epsilon_0$ and $w_0=w^\epsilon_0$.
Theorem \ref{Besovmildsoln} also tells us that $m$ satisfies
\begin{equation} \label{Kpdelestim2}
\norm{m}{K^{\infty,\delta}(T)}+\norm{m}{K^{p,\delta}(T)} \leq C_{q,s}^{(2)}T^{-\frac{\delta}{2}}
\end{equation}
as claimed.
\pagebreak\\
\textbf{Step 2: propagate mild solution properties; $m$ is a Leray-Hopf weak solution:}
Observe that $m_0 \in L^2(\R^3)$ since $$\norm{m_0}{L^2} \leq \norm{u_0}{L^2}+\norm{w_0}{L^2}.$$ We may then use Lemma \ref{contractionLemma} (ii) to propagate this information to deduce that $m$ is the unique Leray-Hopf solution corresponding to $m_0\in \Bes{-1+\frac{3}{p}+\delta}{p}{\infty}(\R^3)\cap L^2(\R^3)$ on $[0,T)$.
\\ First, we show that $m \in L^\infty((0,T);L^2(\R^3))$. Letting $0<t<T$, $a \in L^\infty((0,T);L^2(\R^3))$ and $b \in K^{\infty,\delta}(T)$, we proceed with the same notation as Theorem \ref{Besovmildsoln}:
\begin{align}
\norm{B(a,b)_i(\cdot,t)}{L^2(\R^3)} &\leq \int_0^t \norm{K_{ijl}(\cdot,t-\tau)\ast \left[ a_j(\cdot,\tau)b_l(\cdot,\tau) \right]}{L^2(\R^3)} d\tau
\\ &\leq \int_0^t \norm{K_{ijl}(\cdot,t-\tau)}{L^1(\R^3)} \norm{ a_j(\cdot,\tau)b_l(\cdot,\tau) }{L^{2}(\R^3)} d\tau.
\end{align}
where we used Young's convolution inequality. Using the Oseen kernel estimate (\ref{Oseen Lp bnd}) and applying Hölder's inequality gives:
\begin{align}
\norm{B(a,b)(\cdot,t)}{L^2(\R^3)} &\lesssim  \int_0^t \frac{1}{(t-\tau)^{1/2}} \norm{a(\cdot,\tau)}{L^2(\R^3)}\norm{b(\cdot,\tau)}{L^\infty(\R^3)} d\tau
\\ &\lesssim \norm{a}{L^\infty_tL^2_x}\int_0^t \frac{\tau^{\frac{1}{2}\left(1-\delta \right)}\norm{b(\cdot,\tau) }{L^\infty(\R^3)}}{\tau^{\frac{1}{2}\left(1-\delta \right)}(t-\tau)^{\frac{1}{2}}} d\tau 
\end{align}
Now using the definition of the Kato space norm,
\begin{align}
\norm{B(a,b)(\cdot,t)}{L^2(\R^3)} &\lesssim \norm{a}{L^\infty_tL^2_x}\norm{b}{K^{\infty,\delta}(T)}\int_0^t \frac{1}{\tau^{\frac{1}{2}\left(1-\delta \right)}(t-\tau)^{\frac{1}{2}}} d\tau 
\\ &\lesssim_\delta T^\frac{\delta}{2}\norm{a}{L^\infty_tL^2_x}\norm{b}{K^{\infty,\delta}(T)}  .
\end{align}
Of course, as $B$ is symmetric, we may switch the order of inputs $a$ and $b$ to retrieve the same result. Given that also we have $e^{t\Delta}m_0 \in L^\infty((0,T);L^2(\R^3))$ we are now in a position to apply Lemma \ref{contractionLemma} (ii), which gives us that $m\in L^\infty((0,T);L^2(\R^3))$.

Our next step is to show that $m \otimes m \in L^2((0,T)\times \R^3)$. Indeed, using Hölder's inequality, Lebesgue interpolation between $L^2$ and $L^\infty$ and (\ref{Kpdelestim2}):
\begin{align}
\norm{m \otimes m}{L^2((0,T)\times \R^3)}^2 &\leq \int_0^T \norm{m(\cdot,t)}{L^4(\R^3)}^2 dt
\\ &\leq \int_0^T \norm{m(\cdot,t)}{L^2(\R^3)}\norm{m(\cdot,t)}{L^\infty(\R^3)} dt 
\\&\leq \norm{m}{L^\infty((0,T);L^2(\R^3))}\int_0^T \frac{t^{\frac{1}{2}\left(1-\delta \right)}\norm{m(\cdot,t)}{L^\infty(\R^3)}}{t^{\frac{1}{2}\left(1-\delta \right)}} dt
\\ &\lesssim_\delta T^{\frac{1}{2}\left(1+\delta\right)}\norm{m}{L^\infty((0,T);L^2(\R^3))}\norm{m}{K^{\infty,\delta}(T)}.
\end{align}  
Applying Theorem 2.4.1 of Section IV of Sohr's book \cite{Sohrbook} implies that $m \in L^1_{\text{loc}}([0,T];H^1(\R^3)))$. Then, applying Theorem 2.3.1 of Section IV of Sohr's book \cite{Sohrbook}, we have that $m$ is a weak Leray-Hopf solution on $(0,T)$. 
\\ 
\pagebreak\\\textbf{Step 3: consider the perturbed equation:}
We may then define $w$ on $\R^3 \times (0,T)$ via the relation $w(x,t):=u(x,t)-m(x,t)$.
Then $w$ satisfies the $m$-perturbed Navier-Stokes equations (in the sense of distributions):
\begin{equation}
\partial_tw-\Delta w + (m \cdot \nabla )w + (w \cdot \nabla )m + (w \cdot \nabla )w +\nabla(\pi-\pi_1) =0,
\end{equation}
\begin{equation}
\text{div}( w) = 0,
\end{equation}
\begin{equation}
\lim_{\tau \to 0} \norm{w(\cdot,\tau)-w_0}{L^2}=0
\end{equation}
(where $\pi_1$ is the pressure corresponding to $m$).
By definition of $w$, we see that for any $0<t<T$,
\begin{align}
\frac{1}{2}\norm{w(t)}{L^2_x}^2 &-\frac{1}{2}\norm{w_{0}}{L^2_x}^2 + \int_0^t \norm{\nabla w(s)}{L^2_x}^2 ds 
\\ &=   \frac{1}{2}\norm{u(t)}{L^2_x}^2 - \int_{\R^3}u(x,t)\cdot m(x,t) \: dx +\frac{1}{2}\norm{m(t)}{L^2_x}^2 \\ & \qquad + \int_0^t \norm{\nabla u(s)}{L^2_x}^2 ds -2\int_0^t\int_{\R^3}\nabla u\cdot \nabla m \: dxds +\int_0^t \norm{\nabla m(s)}{L^2_x}^2 ds \\ &\qquad -\frac{1}{2}\norm{u_0}{L^2_x}^2 + \int_{\R^3}u_0\cdot m_0 \: dx -\frac{1}{2}\norm{m_0}{L^2_x}^2 .
\end{align}
Now, as $u$ and $m$ are Leray-Hopf weak solutions on $[0,T]$, we know they satisfy the energy inequality. Furthermore, as $m$ satisfies the Prodi-Serrin condition (see Remark \ref{kato is serrin}), we may use Proposition 14.3 in \cite{Lemarie} on the cross term\footnote{$\int_{\R^3}u(x,t)\cdot m(x,t) \: dx + 2\int_0^t\int_{\R^3}\nabla u\cdot \nabla m \: dxds =  \int_0^t\int_{\R^3} (m\cdot \nabla)u\cdot m + (u\cdot \nabla)u\cdot m \:dxds + \int_{\R^3}u_0\cdot m_0 \: dx $.} and the fact that $u$ and $m$ are weakly divergence-free to see
\begin{align}
&\frac{1}{2}\norm{w(t)}{L^2_x}^2 -\frac{1}{2}\norm{w_{0}}{L^2_x}^2 + \int_0^t \norm{\nabla w(s)}{L^2_x}^2 ds  \leq  \int_0^t \int_{\R^3}(w \cdot \nabla)w \cdot m \: dx ds. \label{Calderoncontrol1}
\end{align}
Now we may estimate, by using Hölder's inequality and Young's inequality,
\begin{align}
\frac{1}{2}\norm{w(t)}{L^2_x}^2 + \int_0^t \norm{\nabla w(s)}{L^2_x}^2 ds &\leq \frac{1}{2}\norm{w_{0}}{L^2_x}^2 + \int_0^t C\norm{m(s)}{L^{\infty}_x}^2\norm{w(s)}{L^2_x}^2 ds \nonumber \\ &+ \frac{1}{2}\int_0^t\norm{\nabla w(s)}{L^2_x}^2 ds. \label{Calderoncontrol2}
\end{align}
Applying Gronwall's inequality yields
\begin{equation} \label{gronreskato2}
\norm{w(t)}{L^2_x}^2 + \int_0^t \norm{\nabla w(s)}{L^2_x}^2 ds \leq \norm{w_{0}}{L^2_x}^2\exp\left(C\int_0^t \norm{m(s)}{L^{\infty}_x}^2 ds \right). 
\end{equation}
But then, by definition of the Kato norm, see that 
\begin{align}
\int_0^T \norm{m(s)}{L^{\infty}_x}^2 ds &\leq \int_0^Ts^{-1+\delta}\norm{m}{K^{\infty,\delta}(T)}^2 ds \label{subcritrequirement}
= C_{q,s} \norm{m}{K^{\infty,\delta}(T)}^2 T^{\delta}.
\end{align}
Combining all of this with (\ref{Kpdelestim2}), we then obtain
\begin{align}
\norm{w}{L^\infty(0,T;L^2(\R^3))}^2 + \int_0^T \norm{\nabla w(s)}{L^2_x}^2 ds &\leq \norm{w_{0}}{L^2_x}^2\exp\left(C_{q,s} (T^{-\frac{\delta}{2}})^2 T^{\delta} \right).
\end{align}
We may then substitute our estimate for $\norm{w_0}{L^2_x}$ from (\ref{westimkato2}) with our specific choice of $\epsilon$ from (\ref{epsilonchoicekato2}) to see
\begin{align} 
\norm{w}{L^\infty(0,T;L^2(\R^3))}^2 + \int_0^T \norm{\nabla w(s)}{L^2_x}^2 ds &\leq C_{q,s}T^\frac{\delta\gamma_2}{\gamma_1}\norm{u_0}{\Bes{s}{q}{\infty}}^{2(\frac{\gamma_2}{\gamma_1}+1)}.
\end{align}
\end{proof}

\subsection{Proof of main result B}
Recall from Section \ref{singularsection}, given a weak Leray-Hopf solution $v$, we denote the singular set at time $T$ by $\sigma(T)$. We denote the $\varrho$-isolated singular set at time $T$ as 
\begin{equation}
\sigma_\varrho(T) := \Big\lbrace x \in \sigma(T) \ \Big\vert \ B_\varrho(x)\cap B_\varrho(y) = \emptyset \quad \forall y \in \sigma(T)\setminus\lbrace x \rbrace \Big\rbrace. \nonumber
\end{equation}
When proving Theorem B, we will rescale the solution to isolate the singular points under consideration by a certain radius. Our key estimate will then come from estimating the number of points in an \textit{isolated} singular set. This key ingredient is made concrete in the following Proposition: 
\begin{Proposition}\label{SinfBesbound} 
Let $\mathcal{M}\geq 1$, $2<q<\infty$, $-1+\frac{3}{q}<s<0$. Let $u$ be a suitable weak Leray-Hopf solution to the Navier-Stokes equations on $\R^3 \times (0, \infty)$, corresponding to divergence-free initial data $u_0$ in $\Bes{s}{q}{\infty}(\R^3) \cap L^2(\R^3)$ such that $$\norm{u_0}{\Bes{s}{q}{\infty}(\R^3)}\leq \mathcal{M}.$$ Then we may bound the number of elements in $\sigma_\frac{1}{2}(1)$: 
\begin{equation}
 \#\sigma_\frac{1}{2}(1) \lesssim_{q,s} \mathcal{M}^{6q-9}. \label{isolatedbound}
\end{equation} 
\end{Proposition}
\begin{proof}[Proof of Proposition \ref{SinfBesbound}]
Let us fix $-1+\frac{3}{q}<s<0$, and set $\varrho:=\frac{1}{2}$. We may label\footnote{We know that the isolated singular set is finite; for example, this may be deduced from Proposition 15.7 of \cite{3DNSE}.} the elements $\sigma_\rho(1)=\lbrace x_1,...,x_N \rbrace$.  We estimate $N$ from above. 
Let us use the previous Lemma \ref{Besovenergy2}: there exist constants $\delta\in \left(0,1-\frac{3}{p}\right)$, $\gamma_1, \gamma_2>0$ all depending on $q$ and $s$ so that we may split $u=m+w$ in such a way that:
\begin{equation} \label{mkatoT}
\norm{m}{K^{p,\delta}(1)}+\norm{m}{K^{\infty,\delta}(1)} \lesssim_{q,s}1,
\end{equation}
\begin{equation} \label{wenergyE}
E:= \sup_{0<t<1} \norm{w(t)}{L^2_x}^2 + \int_0^1 \norm{\nabla w(s)}{L^2_x}^2 ds \lesssim_{q,s}\norm{u_0}{\Bes{s}{q}{\infty}}^{2(\frac{\gamma_2}{\gamma_1}+1)},
\end{equation} 
where $m$ is a weak Leray solution corresponding to partial initial data as described in Lemma \ref{Besovenergy2}. Note that, as $m$ is a smooth suitable weak solution to the Navier-Stokes equations and $u$ is a suitable weak solution to the Navier-Stokes equations, $w$ is a suitable weak solution to the $m$-perturbed Navier-Stokes equations.
It follows by the perturbed version of the Caffarelli-Kohn-Nirenberg Theorem, Theorem \ref{peturbedCKN}, that there exists a constant $\epsilon\in(0,1)$ depending on $q$ and $s$ such that
\begin{equation}
\frac{1}{\varrho^2}\int_{1-\varrho^2}^1\int_{B_\varrho(x_i)} |w|^3 + |\tilde{\pi}|^{3/2} dxdt \geq \epsilon \qquad \text{for} \;\; i=1,\ldots, N,
\end{equation}
where $\tilde{\pi}$ is the perturbed pressure. We may decompose $\tilde{\pi} = \pi_1+\pi_2$, which are defined again by the relations
\begin{align}
-\Delta\pi_1 = \partial_i\partial_j(w_iw_j),
\qquad -\Delta\pi_2 = 2\partial_i\partial_j(m_iw_j).
\end{align}
Now, using just Hölder's inequality, we see that 
\begin{equation}
\int_{1-\varrho^2}^1 \int_{B_\varrho(x_i)} |\pi_2|^{3/2} dxdt \leq \left( \int_{1-\varrho^2}^1 \int_{B_\varrho(x_i)} |\pi_2|^2 dxdt \right)^{3/4} \varrho^{5/4} \omega_3^{1/4},
\end{equation} 
where $\omega_3$ is the volume of the unit ball in $\R^3$.
Thus, we obtain
\begin{equation}
\frac{\epsilon\varrho^2}{2^{1/2}} \leq \int_{1-\varrho^2}^1\int_{B_\varrho(x_i)} |w|^3 dxdt + \int_{1-\varrho^2}^1\int_{B_\varrho(x_i)}|\pi_1|^{3/2}dxdt + \left( \int_{1-\varrho^2}^1 \int_{B_\varrho(x_i)} |\pi_2|^2 dxdt \right)^{3/4} \varrho^{5/4} \omega_3^{1/4}. 
\end{equation}
Therefore, by the pigeonhole principle, either 
\begin{equation}
\int_{1-\varrho^2}^1\int_{B_\varrho(x_i)} |w|^3 dxdt \geq \bar{\epsilon}\varrho^2,
\quad
\text{or} \quad
\int_{1-\varrho^2}^1\int_{B_\varrho(x_i)}|\pi_1|^{3/2}dxdt \geq \bar{\epsilon}\varrho^2,
\end{equation}
\begin{equation}
\text{or}\quad\int_{1-\varrho^2}^1\int_{B_\varrho(x_i)}|\pi_2|^{2}dxdt \geq \frac{\bar{\epsilon}^\frac{4}{3}\varrho}{\omega_3^\frac{1}{3}} ,\qquad \text{where} \qquad \bar{\epsilon} := \dfrac{\epsilon}{3\cdot2^{1/2}}.
\end{equation}
Summing over all $i=1,...,N$, we then see that
\begin{align} \label{Nmin}
N\min\left(\bar{\epsilon}\varrho^2,  \frac{\bar{\epsilon}^\frac{4}{3}\varrho}{\omega_3^\frac{1}{3}}\right) &\leq \sum_{i=1}^N \int_{1-\varrho^2}^1\int_{B_\varrho(x_i)}  |w|^3 + |\pi_1|^{3/2} + |\pi_2|^2dxdt
\\ &\leq \int_{1-\varrho^2}^1\int_{\R^3}  |w|^3 + |\pi_1|^{3/2} + |\pi_2|^2dxdt. \label{Totalestim1}
\end{align}
We will now estimate each of these three terms. 
First, we use Lebesgue interpolation (between $L^2$ and $L^6$) and the Sobolev inequality to obtain:
\begin{equation}
\int_{1-\varrho^2}^1 \norm{w(t)}{L^3_x}^3 dt\leq  \int_{1-\varrho^2}^1 \norm{w(t)}{L^2_x}^{3/2}\norm{w(t)}{L^6_x}^{3/2}dt \lesssim\int_{1-\varrho^2}^1 \norm{w(t)}{L^2_x}^{3/2}\norm{\nabla w(t)}{L^2_x}^{3/2}dt.
\end{equation}
Then using Hölder's inequality and our energy estimate (\ref{wenergyE}) for $w$ , we may estimate
\begin{align}
\int_{1-\varrho^2}^1\int_{\R^3} |w|^3 dxdt &\lesssim \norm{w}{L^{\infty}_tL^2_x}^{3/2}\int_{1-\varrho^2}^1 \norm{\nabla w(t)}{L^2_x}^{3/2}dt \nonumber \\  &\lesssim \varrho^{1/2}\norm{w}{L^{\infty}_tL^2_x}^{3/2} \norm{\nabla w}{L^2_{t,x}}^{3/2}  \lesssim \varrho^{1/2}E^{3/2}. \label{westim}
\end{align}
Next, we use the Calderón–Zygmund Theorem to see that:
\begin{equation}
\norm{\pi_1}{L^{3/2}}^{3/2} \lesssim \norm{w}{L^{3}}^3,\label{pressure1est}
\end{equation}
and 
\begin{align}
\norm{\pi_2}{L^{2}_x} &\lesssim\sum_{i,j}\norm{m_iw_j}{L^{2}_x} \lesssim \norm{m}{L^\infty_x}\norm{w}{L^2_x}, \label{pressure2est}
\end{align}
where we also applied Hölder's inequality in line (\ref{pressure2est}).
\\Using (\ref{pressure1est}) followed by (\ref{westim}), the $\pi_1$ term may be estimated 
\begin{align}
\int_{1-\varrho^2}^1 \norm{\pi_1(t)}{L^{3/2}_x}^{3/2} dt &\lesssim \int_{1-\varrho^2}^1 \norm{w(t)}{L^3_x}^3dt \lesssim \varrho^{1/2}E^{3/2}. \label{pi2estim}
\end{align}
\\For the final term, we may use (\ref{pressure2est}) and our estimates on  (\ref{mkatoT}) and (\ref{wenergyE}) for $m$ and $w$ respectively to see
\begin{align}
\int_{1-\varrho^2}^1\int_{\R^3} |\pi_2|^2dxdt &= \int_{1-\varrho^2}^1 \norm{\pi_2(t)}{L^2_x}^2dt \lesssim \int_{1-\varrho^2}^1 \norm{m}{L^\infty_x}^2\norm{w}{L^2_x}^2dt 
\\ &\lesssim \norm{w}{L^\infty_tL^2_x}^2\int_{0}^1 \norm{m}{K^{\infty,\delta}(1)}^2t^{-(1-\delta)} dt
\\ &\lesssim_{q,s}  E. \label{pi2estim2}
\end{align}
An analysis of the minimum of line (\ref{Nmin}) reveals that  
\begin{subnumcases}{\min{\left(\bar{\epsilon}\varrho^2,\dfrac{\bar{\epsilon}^{\frac{4}{3}}\varrho}{\omega_3^{\frac{1}{3}}} \right)} = }
    \bar{\epsilon}\varrho^2 & $\text{if}\; \varrho\leq \tilde{\epsilon}, $ 
   \\
    \dfrac{\bar{\epsilon}^{\frac{4}{3}}\varrho}{\omega_3^{\frac{1}{3}}} & \text{if}\;$ \varrho\geq \tilde{\epsilon},$ \label{Nmincaseb}
  \end{subnumcases}
  where $\tilde{\epsilon}:=\left( \dfrac{\bar{\epsilon}}{\omega_3}\right)^{1/3}<1$.
As $\tilde{\epsilon}<\frac{1}{2},$ we find ourselves in the second case because $\varrho=\frac{1}{2}$.  
This puts us in a position to combine all these estimates ((\ref{Totalestim1}),(\ref{westim}),(\ref{pi2estim}),(\ref{pi2estim2}) and (\ref{Nmincaseb})) together to see that
\begin{equation} \label{EboundofN}
N \lesssim_{q,s} E^\frac{3}{2} + E.
\end{equation}
To conclude, we use line (\ref{wenergyE}) to substitute $E \lesssim_{q,s}\norm{u_0}{\Bes{s}{q}{\infty}}^{2(\frac{\gamma_2}{\gamma_1}+1)}$. Finally, we use line (\ref{Specificratiosub}) to simplify the exponent, noting that we are in the subcritical case.
\end{proof}
\begin{Remark} \label{isolatedremark}
    This proposition may be easily adapted for all $-1+\frac{2}{q}<s<0$, bounding the number of points in the singular set at any time $T>0$ with any isolation radius $0<\varrho<T^\frac{1}{2}$ in terms of $\mathcal{M}, T$ and $\varrho.$ The exponents will depend on $\gamma_1, \gamma_2$ and $\delta$ from the previous Lemma \ref{Besovenergy2}, which in turn depend on $q$ and $s$. For our purposes and simplicity, we only stated the result in the case above. 
\end{Remark}

With this, we may now move on to proving Theorem B. For convenience, we restate  result here:

\begin{MainTheorem}  \label{sequential scaleinvbnd}
Let $3<q<\infty$, $-1+\frac{3}{q}\leq s<0$ and $M\geq 1$. Let $u: \R^3\times(-1,\infty) \rightarrow \R^3$ be a suitable Leray-Hopf weak solution to the Navier-Stokes equations, where we assume:
\begin{enumerate}
\item[(a)] $u$ first blows up at time $0$;
 \item[(b)] there exists an increasing sequence $t_n \in (-1,0)$ with $t_n\uparrow 0$ such that  
 \begin{align} \label{TypeI time slice assump Besov}
 \sup_{n\in\N}\left[ (-t_n)^{\frac{1}{2}\left(s+1-\frac{3}{q}\right)}\norm{u(\cdot,t_n)}{\Bes{s}{q}{\infty}} \right] \leq M.
 \end{align}

\end{enumerate}  
Then the number of points in the singular set may be estimated as follows: 
\begin{equation}
    \# \sigma(0) \lesssim_{q,s}M^{6q-9}.
\end{equation}
\end{MainTheorem}
\begin{Remark} In (\ref{TypeI time slice assump Besov}), we only consider the subcritical and critical case; this is because the supercritical case implies that there are no singular points at time $t=0$:
Indeed, with the same sequential assumption, the exponent of $-t_n$ in (\ref{TypeI time slice assump Besov}) is negative. Hence, as $t_n \uparrow 0$, $(-t_n)^{\frac{1}{2}\left(s+1-\frac{3}{q}\right)} \to \infty$ and hence we must have $\norm{u(\cdot,t_n)}{\Bes{s}{q}{\infty}} \to 0$. This implies that $u(\cdot,0)=0$. A backwards uniqueness argument (see Theorem 15.4 of \cite{Lermarie21st}) entails that $u$ must be zero on $\R^3 \times (-1,\infty)$, hence it cannot blow-up at time zero.
\end{Remark}
\begin{proof}[Proof of Theorem \ref{sequential scaleinvbnd}]
We adapt the proof of Theorem 2 of \cite{Barkersing21}, which considered the Type-I condition $\sup_{n\in\N}\norm{u(\cdot,t_n)}{L^{3,\infty}(\R^3)}\leq M$ instead of (\ref{TypeI time slice assump Besov}). First, let $\left\lbrace x_1,...,x_N\right\rbrace \subset \sigma(0)$, so that we simply need to bound $N$ from above.
As $t_n\uparrow 0$, there exists $n=n(\left\lbrace x_i\right\rbrace_{i=1}^{N})\in\N$ such that \begin{equation}
\min_{ \left\lbrace (i,j)\in  \left\lbrace 1,...,N\right\rbrace^2\vert i\neq j\right\rbrace}\left\vert x_i-x_j\right\vert \geq 2(-t_n)^\frac{1}{2}.
\end{equation}
Performing the Navier-Stokes rescaling (\ref{NSrescale}) with scale $\lambda:= (-t_n)^\frac{1}{2}<1$, define
\begin{equation}
\left(\tilde{u}(x,t), \tilde{\pi}(x,t)\right) := \left( \lambda u\left(\lambda x, \lambda^2t\right), \lambda^2 \pi\left(\lambda x, \lambda^2t\right)\right)
\end{equation}
so that now $\tilde{u}:\R^3\times\left(\frac{1}{t_n},\infty\right)\rightarrow\R^3$ and $\tilde{\pi}:\R^3\times\left(\frac{1}{t_n},\infty\right)\rightarrow\R$.
Using the rescaling property of the Besov norm (for example, see Remark 2.19 in \cite{Bahouri}) and assumption (b), we observe that: 
\begin{align}
\norm{\tilde{u}(\cdot,-1)}{\Bes{s}{q}{\infty}} &\leq C_s\sup_{n\in\N}\left[ (-t_n)^{\frac{1}{2}\left(s+1-\frac{3}{q}\right)} \norm{u(\cdot,t_n)}{\Bes{s}{q}{\infty}}\right] \leq C_sM, \label{initialtype1bnd}
\end{align}
where $C_s>1$ is a constant depending on $s$.
\\ We now track the location of the singular points $\left\lbrace (x_i,0)\right\rbrace_{i=1}^{N}$ of $u$ under this rescaling: these become $\left\lbrace (y_i,0)\right\rbrace_{i=1}^{N}$ singular points of $\tilde{u}$, where \begin{equation}
y_i:=\frac{x_i}{(-t_n)^\frac{1}{2}} \ \text{for each} \ i\in\lbrace 1,...,N\rbrace.
\end{equation}  
Crucially, we see that \begin{equation}
B_1(y_i) \cap B_1(y_j) = \emptyset \ \text{for all} \  i\neq j \in\lbrace 1,...,N\rbrace.
\end{equation}
Hence, $y_i$ are \textit{isolated} singularities for $\tilde{u}$ (with isolation radius $\geq 1$). Thus, we may apply the previous Proposition \ref{SinfBesbound} to bound $N$ from above. Indeed, by (\ref{isolatedbound}) and using line (\ref{initialtype1bnd}) and setting $\mathcal{M}=C_sM$ yields
\begin{align}
N \leq C_{q,s} M^{6q-9}
\end{align}
as claimed.
\end{proof}
\section{Appendix}
\subsection{Besov Spaces} \label{Besovspaces}
We briefly outline the construction of homogeneous Besov spaces, following the same construction as presented in Chapter 2 of \cite{Bahouri}. First, we define a dyadic partition of unity:
\begin{Proposition} \label{Dyadic} Define the annulus $\mathcal{A}:= \left\lbrace \xi \in \R^d : \frac{3}{4} \leq |\xi | \leq \frac{8}{3} \right\rbrace$. There exists a radial function $\phi \in C_c^\infty(\R^d;[0,1])$ with support in $\mathcal{A}$  satisfying:
\begin{equation}
\sum_{j\in\Z}\phi(2^{-j}\xi) = 1 \qquad \qquad \forall \xi \in \R^d\setminus \left\lbrace  0\right\rbrace \nonumber
\end{equation}
\end{Proposition}
Fixing a choice of $\phi$, the homogeneous Littlewood-Paley projectors $\dot{\Delta}_j$ are Fourier multipliers defined for all $j\in \Z$ by:
\begin{Definition}
$\dot{\Delta}_j f:= \phi(2^{-j}D)f := \mathcal{F}^{-1}\left(\phi(2^{-j}\cdot)\mathcal{F}(f) \right)$ for all tempered distributions $f\in \mathcal{S}'(\R^d)$, where $\mathcal{F}$ is the Fourier transform.
\end{Definition}
The homogeneous Besov spaces are defined using these projectors:
\begin{Definition}
Let $s\in(-\infty,\frac{d}{p})$, $1\leq p,r\leq \infty$.\footnote{These conditions ensure that $\Bes{s}{p}{r}(\R^d)$ is a Banach space. One may also choose $s=\frac{d}{p}$ and $r=1$.} The homogeneous Besov space $\Bes{s}{p}{r}(\R^d)$ is defined as follows: \begin{align} \nonumber
\Bes{s}{p}{r}(\R^d):= \Bigg\lbrace f &\in \mathcal{S}'(\R^d) :\norm{f}{\Bes{s}{p}{r}} := \norm{\left(2^{js}\norm{\dot{\Delta}_j f}{L^p}\right)_{j\in\Z}}{l^r}:= \left(\sum_{j\in\Z} 2^{jrs}\norm{\dot{\Delta}_j f}{L^p}^r \right)^\frac{1}{r} \\ &\text{and} \; \sum_{j\in\Z}\dot{\Delta}_j f \: \text{converges to} \: f \: \text{in the sense of tempered distributions}\Bigg\rbrace. \nonumber
\end{align}
\end{Definition}
\subsection{Classical fixed point/persistency lemma}
\begin{Lemma} \label{contractionLemma}
(i) Let $X$ be a Banach space and let $B:X\times X \rightarrow X$ be a continuous bilinear map such that a constant $\lambda>0$ exists such that
$$\norm{B(x,y)}{X} \leq \lambda\norm{x}{X}\norm{y}{X} \qquad \forall x,y \in X. $$
Let $\alpha>0$ such that  $$\alpha < \dfrac{1}{4\lambda}.$$
\\ Then for any $a$ in the ball $B_{\alpha }(0) \subset X$, there exists a unique $x$ in the ball $B_{2\alpha }(0)$ such that \begin{equation*}
x = a + B(x,x).
\end{equation*}
(ii) Additionally, suppose that $Y \subseteq X$ is another Banach space and $a\in Y $. Suppose also that there exists a constant $0<\gamma<\lambda$ such that:
\begin{align}
    \norm{B(x,y)}{Y} \leq \gamma \norm{x}{X}\norm{y}{Y} \qquad \forall x \in X, \forall y \in Y, \\ \norm{B(y,x)}{Y} \leq \gamma \norm{x}{X}\norm{y}{Y}  \qquad \forall x \in X, \forall y \in Y.
\end{align}
Then the solution $x$ from (i) belongs to $Y$ and satisfies 
\begin{equation}
    \norm{x}{Y} \lesssim \norm{a}{Y}. 
\end{equation} 
\end{Lemma}
See, for example, Lemma A.1 and Lemma A.2 of \cite{GIP03}.
\subsection{Existence of mild soutions: proof of Theorem \ref{Besovmildsoln}}\label{proofofmildtheorem}
\begin{proof}

We are concerned with bounding the bilinear map $B$ to apply Lemma \ref{contractionLemma}. Let $0<t<T$ and $a,b \in K^{p,\delta}(T)$. Using Young's convolution inequality, we estimate:
\begin{align}
\norm{[B(a,b)]_i(\cdot,t)}{L^p(\R^3)} &\leq \int_0^t \norm{K_{ijl}(\cdot,t-\tau)\ast \left[ a_j(\cdot,\tau)b_l(\cdot,\tau) \right] 
 }{L^p(\R^3)} d\tau
\\ &\leq \int_0^t \norm{K_{ijl}(\cdot,t-\tau)
}{L^\alpha(\R^3)} \norm{ a_j(\cdot,\tau)b_l(\cdot,\tau) }{L^{p/2}(\R^3)} d\tau
\end{align}
where $\alpha$ is defined by the relation $1=\frac{1}{p}+\frac{1}{\alpha}.$
Applying Hölder's inequality and the Oseen kernel estimate (\ref{Oseen Lp bnd}) gives:
\begin{align}
\norm{[B(a,b)]_i(\cdot,t)}{L^p(\R^3)} &\leq  \int_0^t \norm{K_{ijl}(\cdot,t-\tau)}{L^\alpha(\R^3)} \norm{a_j(\cdot,\tau)}{L^p(\R^3)}\norm{b_l(\cdot,\tau) }{L^p(\R^3)} d\tau
\\ &\leq \int_0^t \frac{\bar{C_p}}{(t-\tau)^{\frac{1}{2}\left(1+\frac{3}{p}\right)}} \norm{ a(\cdot,\tau)}{L^p(\R^3)}\norm{b(\cdot,\tau) }{L^p(\R^3)} d\tau. \label{oseenestline}
\end{align}
Then, using the definition of the Kato space (Definition \ref{Katospace}), we have
\begin{align}
\norm{B(a,b)(\cdot,t)}{L^p(\R^3)} &\lesssim_{p} \norm{a}{K^{p,\delta}(T)}\norm{b}{K^{p,\delta}(T)} \int_0^t \frac{1}{(t-\tau)^{\frac{1}{2}(1+\frac{3}{p})}\tau^{1-\frac{3}{p}-\delta}}  d\tau
\\ &\lesssim_{p,\delta}\norm{a}{K^{p,\delta}(T)}\norm{b}{K^{p,\delta}(T)}\frac{1}{t^{\frac{1}{2}(1-\frac{3}{p}-2\delta)}}.
\end{align}
Hence, for some constant $\lambda_{p,\delta}>0$ depending on $p$ and $\delta,$
\begin{equation} \label{bilinearKpT}
\norm{B(a,b)}{K^{p,\delta}(T)} \leq \lambda_{p,\delta}T^{\frac{\delta}{2}}\norm{a}{K^{p,\delta}(T)}\norm{b}{K^{p,\delta}(T)}.
\end{equation}
Aiming to use the classical iteration Lemma \ref{contractionLemma}, we know that we require that
\begin{equation} \label{requiremnt for mild}
\norm{e^{t\Delta}v_0}{K^{p,\delta}(T)} < \dfrac{1}{4\lambda_{p,\delta}T^{\frac{\delta}{2}}}.
\end{equation}
Indeed, by the characterisation of the Besov norm by the heat flow (see Proposition \ref{heat characterisation}), we know that 
\begin{align}
\norm{e^{t\Delta}v_0}{K^{p,\delta}(T)}&\leq \tilde{C}_{p,\delta}\norm{v_0}{\Bes{-1+\frac{3}{p}+\delta}{p}{\infty}} . \label{constructocompleto}
\end{align} 
Hence, by setting $C_{p,\delta}^{(1)}:= \dfrac{1}{4\lambda_{p,\delta}\tilde{C}_{p,\delta}}$ and using the hypothesis $T^\frac{\delta}{2}\norm{v_0}{\Bes{-\left(1-\frac{3}{p}-\delta \right)}{p}{\infty}} \leq C_{p,\delta}^{(1)}$, we see that (\ref{requiremnt for mild}) is satisfied. Thus, there exists a mild solution $v\in K^{p,\delta}(T)$ with 
\begin{equation}
\norm{v}{K^{p,\delta}(T)} \leq \dfrac{1}{2\lambda_{p,\delta}T^{\frac{\delta}{2}}}.
\end{equation} 
\\ Finally, we propagate\footnote{We use Theorem \ref{contractionLemma} (ii) - for $\delta>0$, one may simply use the above method with $p=\infty$. However, for $\delta=0$ the constant in line (\ref{bilinearKpT}) degenerates when $p=\infty$, which is why we present the following.} to deduce that $m \in K^{\infty,\delta}(T).$ Indeed, let $a\in K^{p,\delta}(T)$ and $b\in K^{\infty,\delta}(T)$. By a similar application of Young's convolution inequality and Holder's inequality as at the beginning of the proof, we see that
\begin{align}
\norm{[B(a,b)]_i(\cdot,t)}{L^\infty(\R^3)} &\leq  \int_0^t \norm{K_{ijl}(\cdot,t-\tau)}{L^\beta(\R^3)} \norm{a_j(\cdot,\tau)}{L^p(\R^3)}\norm{b_l(\cdot,\tau) }{L^\infty(\R^3)} d\tau
\\ &\lesssim_p \int_0^t \frac{1}{(t-\tau)^{\frac{1}{2}\left(1+\frac{3}{p}\right)}} \norm{ a(\cdot,\tau)}{L^p(\R^3)}\norm{b(\cdot,\tau) }{L^\infty(\R^3)} d\tau \label{oseenestline2}
\end{align}
where $\beta$ is defined by the relation $1=\frac{1}{p}+\frac{1}{\beta}.$ Then
\begin{align}
\norm{[B(a,b)](\cdot,t)}{L^\infty(\R^3)} &\lesssim_p  \int_0^t \frac{1}{(t-\tau)^{\frac{1}{2}\left(1+\frac{3}{p}\right)}(\tau)^{1-\delta-\frac{3}{2p}}} \norm{ a}{K^{p,\delta}(T)}\norm{b}{K^{\infty,\delta}(T)} d\tau. 
\\ &\lesssim_{p,\delta}  \frac{1}{t^{\frac{1}{2}-\delta}}\norm{ a}{K^{p,\delta}(T)}\norm{b}{K^{\infty,\delta}(T)}.
\end{align}
Therefore, for some constant $\gamma_{p,\delta}>0$ depending on $p$ and $\delta$, we have:
\begin{align}
    \norm{B(a,b)}{K^{\infty,\delta}(T)}\leq \gamma_{p,\delta}T^\frac{\delta}{2}\norm{ a}{K^{p,\delta}(T)}\norm{b}{K^{\infty,\delta}(T)}.
\end{align}
Without loss of generality, we may assume $\lambda_{p,\delta}>\gamma_{p,\delta}$. This allows us to apply Theorem \ref{contractionLemma} (ii) and the Besov space embedding $\Bes{-1+\frac{3}{p}+\delta}{p}{\infty}(\R^3) \hookrightarrow \Bes{-1+\delta}{\infty}{\infty}(\R^3)$ (see Proposition \ref{Besov embeddings}) to deduce that $v \in K^{p,\delta}(T)\cap K^{\infty, \delta}(T)$ with
\begin{align}
    \norm{v}{K^{\infty, \delta}(T)} \leq \norm{e^{t\Delta}v_0}{K^{\infty, \delta}(T)} \lesssim_\delta\norm{v_0}{\Bes{-1+\delta}{\infty}{\infty}} \lesssim_{p,\delta}\norm{v_0}{\Bes{-1+\frac{3}{p}+\delta}{p}{\infty}} \lesssim_{p,\delta}T^{-\frac{\delta}{2}}.
\end{align}
Then we take $C^{(2)}_{p,q}$ larger than $\frac{1}{2\lambda_{p,\delta}}$ and the implicit constant depending on $p$ and $\delta$ in the line above to conclude.
\end{proof}

\subsubsection*{Acknowledgements} The author thanks Tobias Barker for his supervision and many helpful discussions. This work is supported by Raoul \& Catherine Hughes (Alumni funds) and the University Research Studentship award EH-MA1333.

\bibliographystyle{plain}
\small
\nocite{*} 
\bibliography{Bibliography}

\end{document}